\documentclass[]{interact}
\usepackage{amsmath}
\usepackage[utf8]{inputenc}
\usepackage[T1]{fontenc}
\usepackage{latexsym}
\usepackage{amssymb}
\usepackage{lmodern}
\usepackage{fix-cm}
\usepackage{url}
\usepackage{hyperref}
\usepackage{graphicx}
\usepackage{tabularx}
\usepackage{array}
\usepackage{mdwmath}
\usepackage{mdwtab}
\usepackage{float}

\usepackage[cmyk]{xcolor}
\usepackage{stmaryrd}
\usepackage{tikz}

\usepackage{makecell}
\usepackage{diagbox}

\usepackage{graphicx}
\usepackage{tabularx}
\usepackage{array}
\usepackage{lineno}
\usepackage{booktabs}
\usepackage{eurosym}

\newcommand{\pc}{\mathbf{P}}

\newcommand{\pp}{\mathbb{P}}

\newcommand{\ml}{\mathfrak{L}}
\newcommand{\mm}{\blacklozenge}

\theoremstyle{plain}
\newtheorem{theorem}{Theorem}[section]

\newtheorem{proposition}[theorem]{Proposition}

\theoremstyle{definition}
\newtheorem{definition}[theorem]{Definition}
\newtheorem{example}[theorem]{Example}

\theoremstyle{remark}

\usepackage[natbibapa,nodoi]{apacite}
\begin{document}


\title{Dialectical Rough Sets, Parthood and Figures of Opposition-1}

\author{
\name{A. Mani \thanks{CONTACT A. Mani Email: a.mani.cms@gmail.com ; Web: \url{http://www.logicamani.in}}}
\affil{Department of Pure Mathematics, University of Calcutta\\
International Rough Set Society\\
9/1B, Jatin Bagchi Road, Kolkata-700029, India
}}

\maketitle
\begin{abstract}
In one perspective, the main theme of this research revolves around the inverse problem in the context of general rough sets that concerns the existence of rough basis for given approximations in a context. Granular operator spaces and variants were recently introduced by the present author as an optimal framework for anti-chain based algebraic semantics of general rough sets and the inverse problem. In the framework, various sub-types of crisp and non-crisp objects are identifiable that may be missed in more restrictive formalism. This is also because in the latter cases concepts of complementation and negation are taken for granted - while in reality they have a complicated dialectical basis. This motivates a general approach to dialectical rough sets building on previous work of the present author and figures of opposition. In this paper dialectical rough logics are invented from a semantic perspective, a concept of dialectical predicates is formalised, connection with dialetheias and glutty negation are established, parthood analyzed and studied from the viewpoint  of classical and dialectical figures of opposition by the present author. Her methods become more geometrical and encompass parthood as a primary relation (as opposed to roughly equivalent objects) for algebraic semantics. 
\end{abstract}

\section{Introduction}

It is well known that sets of rough objects (in various senses) are quasi or partially orderable. Specifically in classical or Pawlak rough sets \cite{zpb}, the set of roughly equivalent sets has a quasi Boolean order on it while the set of rough and crisp objects is Boolean ordered.  In the classical semantic domain or classical meta level, associated with general rough sets, the set of crisp and rough objects is quasi or partially orderable. Under minimal assumptions on the nature of these objects, many orders with rough ontology can be  associated - these necessarily have to do with concepts of discernibility. Concepts of rough objects, in these contexts, depend additionally on approximation operators and granulations used. Many generalizations of classical rough sets from granular perspectives have been studied in the literature. For an overview the reader is referred to \cite{am501}. These were part of the motivations for the invention of the concept of granular operator spaces by the present author in \cite{am6999} and developed further in \cite{am9114,am6900,am9699}. In the paper, it is shown that collections of mutually distinct objects (relative to some formulas) form antichains. Models of these objects, proposed in the paper, have been associated with deduction systems in \cite{am9114}. Collections of antichains can be partially ordered in many ways. Relative to each of these partial orders, maximal antichains can be defined. Maximal antichains are sets of mutually distinct objects and any of its supersets is not an antichain. In one sense, this is a way of handling roughness. To connect these with other rough objects, various kinds of negation-like operations (or generalizations thereof) are of interest. Such operations and predicates are of interest when rough parthoods are not partial or quasi orders \cite{am240,lp2011,am6900}.

Given some information about approximations in a context, the problem of providing a rough semantics is referred to as an \emph{inverse problem}. This class of problems was introduced by the present author in \cite{am3} and has been explored in \cite{am240,am5019} in particular. The solution for classical rough sets may be found in \cite{bc1}. For more practical contexts, abstract frameworks like \emph{rough Y-systems} \textsf{RYS} \cite{am240} are better suited for problem formulation. However, good semantics may not always be available for \textsf{RYS}. It is easier to construct algebraic semantics over granular operator spaces and generalizations thereof \cite{am9114,am6900} as can be seen in the algebraic semantics of \cite{am9114} that involve distributive lattices without universal negations. Therefore, these frameworks are better suited for solving inverse problems and the focus of this paper will be limited to these.  

The square of opposition and variants, in modern interpretation, refer to the relation between quantified or modal sentences in different contexts \cite{ip,fs2012}. These have been considered in the context of rough sets in \cite{cd9} from a set theoretical view of approximations of subsets (in some rough set theories). The relation of parthood in the context of general rough sets with figures has not been investigated in the literature.This is taken up in the present research paper (but from a semantic perspective). Connections with dialectical predication and other kinds of opposition are also taken up in the light of recent developments on connections of para-consistency and figures of opposition. 

At another level of conception, in the classical semantic domain, various sub-types of objects relate to each other in specific ways through ideas of approximations (the origin of these approximations may not be known clearly). These ways are shown to interact in dialectical ways to form other semantics under some assumptions. The basic structural schema can be viewed as a generalization of ideas like the square and hexagon of opposition - in fact as a combination of  dialectical oppositions involved. This aspect is explained and developed in detail in the present paper. These are also used for introducing new methods of counting and semantics in a forthcoming paper by the present author. In \cite{am699}, a dialectical rough set theory was developed by the present author using a specific concept of dialectical contradiction that refers to both rough and classical objects. Related parthood relations are also explored in detail with a view to construct possible models (semantics).

All these are part of a unified whole - the inverse problem and possible solution strategies in all its generality. Granular operator spaces and variants \cite{am6999,am6900} are used as a framework for the problem, and as these have been restricted by the conditions imposed on the approximations, it makes sense to speak of a \emph{part of a unified whole}. \emph{Importantly no easy negation like  operators are definable/defined in the framework and this is also a reason for exploring/identifying dialectical negations}. So at one level the entire paper is a contribution to possible solutions of the inverse problem and usable frameworks for the same. This also involves the invention of a universal dialectical negation (or opposition) from a formal perspective on the basis of diverse philosophical and practical considerations. 

Semantic domains (or domains of discourse) associated with rough objects, operations and predications have been identified by the present author in \cite{am240,am3930}, \cite{am9114,am9501,am501}. The problem of defining rough objects that permit reasoning about both intensional and extensional aspects posed in \cite{mkcfi2016} corresponds to identification of suitable semantic domains. This problem is also addressed in this research (see Sec\ref{cer}). 

The questions and problems that are taken up in this research paper and solved to varying extents also include the following:
\begin{description}
\item[1]{What may be a proper formalization of a dialectical logic and dialectical opposition?
\begin{itemize}
\item {What is the relation between dialetheias, truth gluts and dialectical contradiction?}
\item {Should the objects formalised by the logic be interpreted as state transitions?}
\item {How do dialetheias and dialectical contradiction differ?}
\end{itemize}}
\item[2]{Do paraconsistent logics constitute a proper formalization of the philosophical intent in Hegelian and Marxist dialectics?}
\item[3]{What is the connection between parthood in rough contexts and possible dialectical contradictions?}
\item[4]{What is a rough dialectical semantics and is every parthood based rough semantics a dialectical one?}
\item[5]{How does parthood relate to figures of opposition?}
\item[6]{What is a useful representation of rough objects that addresses the concerns of \cite{mkcfi2016}?}
\end{description}

This paper is structured as follows. The next subsection includes background material on rough concepts, posets, granules and parthood. In the next section, the superiority of granular operator spaces over property systems is explained. Dialectical negation and logics are characterised from a critical perspective in the third section. In the following section, dialectical rough logics are developed and related parthoods are explored. Many examples are provided in the context of the semantic framework used in the fifth section. The sixth section is about figures of dialectical opposition generated by few specific parthood related statements in rough sets and a proposal for handling pseudo gluts. Some directions are also provided in the section. The reader can possibly omit some of the philosophical aspects of the section on \emph{dialectical negation} on a first reading. But a proper understanding of the section would be useful for applications.

\subsection{Background}

In quasi or partially ordered sets, sets of mutually incomparable  elements are called \emph{antichains}. Some of the basic properties may be found in \cite{gg1998,koh}. The possibility of using antichains of rough objects for a possible semantics was mentioned in \cite{am9501,am6000,am3690} and has been developed subsequently in \cite{am6999,am9114}. The semantics invented in the paper is applicable for a large class of operator based rough sets including specific cases of rough Y- systems (\textsf{RYS}) \cite{am240} and other general approaches like \cite{cd3,cc5,yy9,it2}. In \cite{cd3,cc5,gc2018}, negation like operators are assumed in general and these are not definable operations in terms of order related operations in a algebraic sense (Kleene negation is also not a definable operation in the situation). 

For basics of rough sets, the reader is referred to \cite{zpb,ppm2,lp2011}. Distinct mereological approaches to rough sets can be found in \cite{ps3,lp2011,am9114}.  

Many concepts of \emph{lower} and \emph{upper approximation operators} are known in the literature. Relevant definitions are fixed below.

If $S$ is any set (in ZFC), then a \emph{lower approximation operator} $l$ over $S$ is be a map $l: \wp(S) \longmapsto \wp (S)$ that satisfies:
\begin{align}
(\forall x \in \wp(S) ) \,\neg (x \subset x^l) \tag{non-increasing}\\
(\forall x \in \wp(S) )\, x^l = x^{ll} \tag{idempotence}\\
(\forall a, b \in \wp(S) ) \, (a \subseteq b \longrightarrow a^l \subseteq b^l ) \tag{monotonicity}
\end{align}

In the literature on rough sets many variants of the above are also referred to as lower approximations because the concept has to do with what one thinks a \emph{lower approximation} ought to be in the context under consideration. Over the same set, an \emph{upper approximation operator} $u$ shall be a map  $u : \wp (S) \longmapsto \wp (S)$ that satisfies: 
\begin{align}
(\forall x \in \wp(S) ) \, x \subseteq x^u \tag{increasing}\\
(\forall a, b \in \wp(S) ) \, (a \subseteq b \longrightarrow a^u \subseteq b^u ) \tag{monotonicity}
\end{align}
All of these properties hold (in a non-trivial sense) in proto-transitive rough sets \cite{am9501}, while weaker properties hold for lower approximations in esoteric rough sets \cite{am24,am501}. Conditions involving both $l$ and $u$ have been omitted for simplicity.

In some practical contexts, lower and upper approximation operators may be partial and $l$ (respectively $u$) may be defined on a subset $\mathcal{S}\subset \wp(S)$ instead. In these cases, the partial operation may not necessarily be easily completed. Further the properties attributed to approximations may be a matter of discovery. These cases fall under the general class of \emph{inverse problems} \cite{am3,am240} where the goal is to see whether the approximations originate or fit a rough evolution (process). More details can be found in the next section.

An element $x\in\wp{S}$ will be said to be \emph{lower definite} (resp. \emph{upper definite}) if and only if $x^l = x$ (resp. $x^u = x$) and \emph{definite}, when it is both lower and upper definite. In general rough sets, these ideas of definiteness are insufficient as it can happen that upper approximations of upper approximations are still not upper definite. 

The concept of a Rough Y-System \textsf{RYS} was introduced by the present author in \cite{am240,am99} and refined further in \cite{am1800} and her doctoral thesis as a very general framework for rough sets from an axiomatic granular perspective. The concept is not used in an essential way in the present paper and the reader may skip the few remarks that involve them. In simplified terms, it is a model of any collection of rough/crisp objects with approximation operators and a binary parthood predicate $\pc$ as part of its signature.   

Possible concepts of rough objects considered in the literature include the following: 
\begin{itemize}
\item {\textsf{Non definite subsets} of $S$; formally, $x$ is a rough object if and only if $x^l \neq x^u$. }
\item {\textsf{Pairs of definite subsets} of the form $(a, b)$ that satisfy $a\subseteq b$.}
\item {\textsf{Pairs of subsets} of the form $(x^l ,x^u)$.}
\item {Sets in an \textsf{interval of the form} $(x^l, x^u)$; formally, }
\item {Sets in an \textsf{interval of the form} $(a, b)$ satisfying $a\subseteq b$ with $a, b$ being definite subsets.}
\item {Higher order intervals bounded by definite subsets \cite{am240}.}
\item {\textsf{Non-definite elements in a RYS} \cite{am240}; formally, those $x$ satisfying $\neg \pc x^u x^l $ are rough objects.}
\end{itemize}

In general, a given set of approximations may be compatible with multiple concepts of definite and rough objects - the eventual choice depends on one's choice of semantic domain. A detailed treatment, due to the present author, can be found in \cite{am240,am501}. In \cite{mkcfi2016}, some of these definitions of rough objects are regarded as imperfect for the purpose of expressing both multiple extensions of concepts and rough objects - this problem relates to representation (given a restricted view of granularity) within the classical semantic domain of \cite{am240}.

\textsf{Concepts of representation of objects necessarily relate to choice of semantic frameworks. But in general, order theoretic representations are of interest in most contexts. In operator centric approaches, the problem is also about finding ideal representations.}

In simple terms, \emph{granules} are the subsets that generate approximations and \emph{granulations} are the collections of all such granules in the context. In the present author's view, there are at least three main classes of granular computing (five classes have been considered by her in \cite{am501}). The three main classes are 

\begin{itemize}
\item{\textsf{Primitive Granular Computing Processes (PGCP)}: in which the problem requirements are not rigid, effort on formalization is limited, scope of abstraction become limited and concept of granules are often vague, though they maybe concrete or abstract (relative to all materialist viewpoints).} 
\item{\textsf{Classical Granular Computing Paradigm ( CGCP)}: The precision based \emph{classical granular computing paradigm}, traceable to Moore and Shannon's paper \cite{sha56}, is commonly understood as the granular computing paradigm (The reader may  note that the idea is vaguely present in \cite{sha48}). CGCP has since been adapted to fuzzy and rough set theories in different ways. An overview is considered in \cite{tyl}. In CGCP, granules may exist at different levels of precision and granulations (granules at specific levels or processes) form a hierarchy that may be used to solve the problem in question.}
\item{\textsf{Axiomatic Approach to Granularity (AAG)}: The axiomatic approach to granularity, initiated in \cite{am99}, has been developed by the present author in the direction of contamination reduction in \cite{am240}. The concept of admissible granules, described below, was arrived in the latter paper and has been refined/simplified subsequently in \cite{am3930,am1800,am9114}.  From the order-theoretic algebraic point of view, the deviation is in a very new direction relative the precision-based paradigm. The paradigm shift includes a new approach to measures.}
\end{itemize}
Historical details can be found in a section in \cite{am501,am9501}.

Granular operator spaces, a set framework with operators introduced by the present author in \cite{am6999}, will be used as considerations relating to antichains will require quasi/partial orders in an essential way. The evolution of the operators need not be induced by a cover or a relation (corresponding to cover or relation based systems respectively), but these would be special cases. The generalization to some rough Y-systems \textsf{RYS} (see \cite{am240} for definitions), will of course be possible as a result. 

\begin{definition}\label{gos}
A \emph{Granular Operator Space}\cite{am6999} $S$ is a structure of the form $S\,=\, \left\langle \underline{S}, \mathcal{G}, l , u\right\rangle$ with $\underline{S}$ being a set, $\mathcal{G}$ an \emph{admissible granulation}(defined below) over $S$ and $l, u$ being operators $:\wp(\underline{S})\longmapsto \wp(\underline{S})$ ($\wp(\underline{S})$ denotes the power set of $\underline{S}$) satisfying the following ($\underline{S}$ will be replaced with $S$ if clear from the context. \textsf{Lower and upper case alphabets will both be used for subsets} ):

\begin{align*}
a^l \subseteq a\,\&\,a^{ll} = a^l \,\&\, a^{u} \subset a^{uu}  \\
(a\subseteq b \longrightarrow a^l \subseteq b^l \,\&\,a^u \subseteq b^u)\\
\emptyset^l = \emptyset \,\&\,\emptyset^u = \emptyset \,\&\,\underline{S}^{l}\subseteq S \,\&\, \underline{S}^{u}\subseteq S.
\end{align*}

Here, \emph{admissible granulations} are granulations $\mathcal{G}$ that satisfy the following three conditions\\ ($t$ is a term operation formed from the set operations $\cup, \cap, ^c, 1, \emptyset$):

\begin{align*}
(\forall a \exists
b_{1},\ldots b_{r}\in \mathcal{G})\, t(b_{1},\,b_{2}, \ldots \,b_{r})=a^{l} \\
\tag{Weak RA, WRA} \mathrm{and}\: (\forall a)\,(\exists
b_{1},\,\ldots\,b_{r}\in \mathcal{G})\,t(b_{1},\,b_{2}, \ldots \,b_{r}) =
a^{u},\\
\tag{Lower Stability, LS}{(\forall b \in
\mathcal{G})(\forall {a\in \wp(\underline{S}) })\, ( b\subseteq a\,\longrightarrow\, b \subseteq a^{l}),}\\
\tag{Full Underlap, FU}{(\forall
a,\,b\in\mathcal{G})(\exists
z\in \wp(\underline{S}) )\, a\subset z,\,b \subset z\,\&\,z^{l} = z^{u} = z,}
\end{align*}
\end{definition}

\begin{flushleft}
\textbf{Remarks}: 
\end{flushleft}
\begin{itemize}
\item {The concept of admissible granulation was defined for \textsf{RYS} in \cite{am240} using parthoods instead of set inclusion and relative to \textsf{RYS}, $\pc = \subseteq$, $\pp = \subset$. }
\item {The conditions defining admissible granulations mean that every approximation is somehow representable by granules in a set theoretic way, that granules are lower definite, and that all pairs of distinct granules are contained in definite objects.}
\end{itemize}

On $\wp(\underline{S})$, the relation $\sqsubset$ is defined by \begin{equation}A \sqsubset B \text{ if and only if } A^l \subseteq B^l \,\&\, A^u \subseteq B^u.\end{equation} The rough equality relation on $\wp(\underline{S})$ is defined via \[A\approx B \text{ if and only if } A\sqsubset B  \, \&\,B \sqsubset A.\] 

Regarding the quotient $\wp(\underline{S})|\approx$ as a subset of $\wp(\underline{S})$, an order $\Subset$ can be defined as follows: \begin{equation}\alpha \Subset \beta \text{ if and only if } \alpha^l \subseteq \beta^l \,\&\, \alpha^u \subseteq \beta^u.\end{equation} Here $\alpha^l$ is being interpreted as the lower approximation of $\alpha$ and so on. $\Subset$ will be referred to as the \emph{basic rough order}.

\begin{definition}
By a \emph{roughly consistent object} will be meant a set of subsets of $\underline{S}$ with mutually identical lower and upper approximations respectively. In symbols $H$ is a roughly consistent object if it is of the form  $H = \{A ; (\forall B\in H)\,A^l =B^l, A^u = B^u \}$. The set of all roughly consistent objects is partially ordered by the inclusion relation. Relative this maximal roughly consistent objects will be referred to as \emph{rough objects}. By \emph{definite rough objects}, will be meant rough objects of the form $H$ that satisfy 
\[(\forall A \in H) \, A^{ll} = A^l \,\&\, A^{uu} = A^{u}. \] 
\end{definition}

However, this definition of rough objects will not necessarily be followed in this paper.

\begin{proposition}
$\Subset$ is a bounded partial order on $\wp(\underline{S})|\approx$. 
\end{proposition}
\begin{proof}
Reflexivity is obvious.  If $\alpha \Subset \beta$ and $\beta \Subset \alpha$, then it follows that $\alpha^l = \beta^l$ and $\alpha^u = \beta^u$ and so antisymmetry holds. 

If $\alpha \Subset \beta$, $\beta \Subset \gamma$, then the transitivity of set inclusion induces transitivity of $\Subset$.
The poset is bounded by $0 = (\emptyset , \emptyset)$ and $1 = (S^l , S^u)$. Note that $1$ need not coincide with $(S, S)$. 

\end{proof}

The concept of \emph{generalised granular operator spaces} has been introduced in \cite{am9114,am6900} as a proper generalization of that of granular operator spaces. The main difference is in the replacement of $\subset$ by arbitrary \emph{part of} ($\pc$) relations in the axioms of admissible granules and inclusion of $\pc$ in the signature of the structure.
\begin{definition}
A \emph{General Granular Operator Space} (\textsf{GSP}) $S$ is a structure of the form $S\,=\, \left\langle \underline{S}, \mathcal{G}, l , u, \pc \right\rangle$ with $\underline{S}$ being a set, $\mathcal{G}$ an \emph{admissible granulation}(defined below) over $S$, $l, u$ being operators $:\wp(\underline{S})\longmapsto \wp(\underline{S})$ and $\pc$ being a definable binary generalized transitive predicate (for parthood) on $\wp(\underline{S})$ satisfying the same conditions as in Def.\ref{gos} except for those on admissible granulations (Generalised transitivity can be any proper nontrivial generalization of parthood (see \cite{am3690}). $\pp$ is  proper parthood (defined via $\pp ab$ if and only if $\pc ab \,\&\,\neg \pc ba$) and $t$ is a term operation formed from set operations):

\begin{align*}
(\forall x \exists
y_{1},\ldots y_{r}\in \mathcal{G})\, t(y_{1},\,y_{2}, \ldots \,y_{r})=x^{l} \\
\tag{Weak RA, WRA} \mathrm{and}\: (\forall x)\,(\exists
y_{1},\,\ldots\,y_{r}\in \mathcal{G})\,t(y_{1},\,y_{2}, \ldots \,y_{r}) =
x^{u},\\
\tag{Lower Stability, LS}{(\forall y \in
\mathcal{G})(\forall {x\in \wp(\underline{S}) })\, ( \pc yx\,\longrightarrow\, \pc yx^{l}),}\\
\tag{Full Underlap, FU}{(\forall
x,\,y\in\mathcal{G})(\exists
z\in \wp(\underline{S}) )\, \pp xz,\,\&\,\pp yz\,\&\,z^{l} = z^{u} = z,}
\end{align*}
\end{definition}

It is sometimes more convenient to use only sets and subsets in the formalism as these are the kinds of objects that may be observed by agents and such a formalism would be more suited for reformulation in formal languages. For this reason higher order variants of general granular operator spaces have been defined in \cite{am9006} by the present author. A detailed account can be found in \cite{am501}. 

In a partially ordered set \emph{chains} are subsets in which any two elements are comparable. \emph{Antichains}, in contrast, are subsets in which no two elements are comparable. Singletons are both chains and antichains.

\subsubsection{Rough Sets and Parthood }

It is necessary to clarify the nature of parthood even in set-theoretic structures like granular operator spaces.
The restriction of the parthood relation to the case when the first argument is a granule is particularly significant. The theoretical assumption that \textsf{objects} are determined by their parts, and specifically by granules, may not be reasonable when knowledge of the context is evolving. Indeed, in this situation:
\begin{itemize}
\item {granulation can be confounded by partial nature of information and noise,}
\item {knowledge of all possible granulations may not be possible and the chosen set of granules may not be optimal for handling partial information, and }
\item {the process has strong connections with apriori understanding of the objects in question.}
\end{itemize}

\section{Types of Preprocessing and Ontology}

Information storage and retrieval systems (also referred to as information tables, descriptive systems, knowledge representation system) are basically representations of structured data tables. These have often been referred to as  \emph{information systems} in the rough set literature - the term means \emph{an integrated heterogeneous system that has components for collecting, storing and processing data} in closely allied fields like artificial intelligence, database theory and machine learning. So it makes sense to avoid referring to \emph{information tables} as \emph{information systems} \cite{cd2017}.

Information tables are associated with only some instances of data in real life. In many cases, such association may not be useful enough in the first place. \emph{It is also possible that data collected in a context is in mixed form with some information being in information table form and some in the form of \emph{relevant approximations} or all the main data is in terms of approximations}. The \emph{inverse problem}, introduced by the present author in \cite{am3} and subsequently refined in \cite{am240}, seeks to handle these types of situations. Granular operator spaces and higher order variants studied by the present author in \cite{am6999,am9114,am9006} are important structures that can be used for its formulation. In simple terms, the problem is a generalization of the duality problem which may be obtained by replacing the semantic structures with parts thereof. In a mathematical sense, this generalization may not be proper (or conservative) in general. 

The basic problem is \emph{given a set of approximations, similarities and some relations about the objects,
find an information system or a set of approximation spaces that fits the available information according to a rough procedure}. In this formalism, a number of information tables or approximation systems along with rough procedures may fit in. Even when a number of additional conditions like lattice orders, aggregation and commonality operations are available, the problem may not be solvable in a unique sense. Negation-like operations and generalizations thereof can play a crucial role in possible formalisms. In particular, the ortho-pair approach \cite{gcd2018} to rough sets relies on negations in a very essential way. 

It is also necessary to concentrate on the evolution of negation like operations or predicates at a higher level - because it is always necessary to explain the choice of operations used in a semantics. 
The following example (Ex.\ref{exa1}) from \cite{am9114} illustrates the inverse problem. Some ways of identifying generalised negation-like predicates are shown in Ex.\ref{badminton} -

\begin{example}\label{exa1}
This example has the form of a narrative in \cite{am9114} that gets progressively complex. It has been used to illustrate a number of computational contexts in the paper.

Suppose Alice wants to purchase a laptop from an on line store for electronics. Then she is likely to be confronted by a large number of models and offers from different manufacturers and sellers. Suppose also that the she is willing to spend less than \officialeuro $x$ and is pretty open to considering a number of models. This can happen, for example, when she is just looking for a laptop with enough computing power for her programming tasks.   

This situation may appear to have originated from information tables with complex rules in columns for decisions and preferences. Such tables are not information systems in the proper sense. Computing power, for one thing, is a context dependent function of CPU cache memories, number of cores, CPU frequency, RAM, architecture of chip set, and other factors like type of hard disk storage. 

\begin{proposition}
The set of laptops $\mathbb{S}$ that are priced less than \officialeuro $x$ can be totally quasi-ordered. 
\end{proposition}
\begin{proof}
Suppose $\prec$ is the relation defined according to $a \prec b$ if and only if price of laptop $a$ is less than or equal to that of laptop $b$. Then it is easy to see that $\prec$ is a reflexive and transitive relation.
If two different laptops $a$ and $b$ have the same price, then $a \prec b$ and $b\prec a$ would hold.
So $\prec$ may not be antisymmetric.

\end{proof}

Suppose that under an additional constraint like CPU brand preference, the set of laptops becomes totally ordered. That is under a revised definition of $\prec$ of the form: $a \prec b$ if and only if price of laptop $a$ is less than that of laptop $b$ and if the prices are equal then CPU brand of $b$ must be preferred over $a$'s.

Suppose now that Alice has more knowledge about a subset $C$ of models in the set of laptops $\mathbb{S}$. Let these be labeled as \emph{crisp} and let the order on $C$ be $\prec_{|C}$. Using additional criteria, rough objects can be indicated. Though lower and upper approximations can be defined in the scenario, the granulations actually used are harder to arrive at without all the gory details.

This example once again shows that granulation and construction of approximations from granules may not be related to the construction of approximations from properties in a cumulative way. 
\end{example}

In \cite{am9114}, it is also shown that the number of data sets, of the form mentioned, that fit into a rough scheme of things are relatively less than the number of those that do not fit. Many combinatorial bounds on the form of rough object distribution are also proved in the paper.

Examples of approximations that are not rough in any sense are common in misjudgments and irrational reasoning guided by prejudice. So solutions to the problem can also help in judging the irrationality of reasoning and algorithms in different contexts.

\begin{example}\label{badminton}
Databases associated with women badminton players have the form of multi-dimensional information tables about performance in various games, practice sessions, training regimen and other relevant information. Video data about all these would also be available. Typically these are updated at specific intervals of time. Players tend to perform better under specific conditions that include the state of the game in progress. They may also be able to raise the level of their game under specific kinds of stress - this involves dynamic learning. Additional information of the form can be expressed in terms of approximations, especially when the associations are not too perfect. Thus, a statement like 
\begin{quote}
Player $A$ is likely to perform at least as well as player $B$ in playing conditions $C$ 
\end{quote}
can be translated into $B^l \leq A^l $ where the approximations refer the specific property under consideration.
\emph{But this information representation has no obvious rough set basis associated} and falls under the inverse problem, where the problem is of explaining the approximations from a rough perspective/basis. 

Consider a pair $(a, b)$ with 
Updates to the database can also be described through generalised negation-like predicates in the context 
\end{example}

\subsection{Granular Operator Spaces and Property Systems}

Limitations of the property system approach are mentioned in this subsection 

In general, data can be presented in real life partly in terms of approximations and in the object-attribute-value way of representing things (For those that like statistics, the collection of instances of the sentence has nice statistical properties). Such contexts were never intended to be captured through property systems or related basic constructors (see \cite{ppm2,yy2012c,gdu}). In particular, the examples of \cite{pp20} \emph{are abstract and  the possible problems with basic constructors (when viewed from the perspective of approximation properties satisfied) are issues relating to construction - empirical aspects are missed}. 

\begin{definition}
A \emph{property system}\cite{pp20,ppm2,gs1987,gs1999} $\Pi$ is a triple of the form \[\left\langle U, P, R\right\rangle\] with $U$ being a universe of objects, $P$ a set of properties, and $R\subseteq U\times P$ being a \emph{manifestation relation} subject to the interpretation object $a$ has property $b$ if and only if  $(a, b )\in R$. When $P = U $, then $\Pi$ is said to be a \emph{square relational system} and $\Pi$ then can be read as a Kripke model for a corresponding modal system. 

On property systems, basic constructors that may be defined for $A\subseteq U $ and $B\subseteq P$ are 
\begin{align}
<i> :\wp (U )\longmapsto \wp (P);\; <i>(A) = \{ h : (\exists g\in A)\, (g, h) \in R\}\\
<e>: \wp(P ) \longmapsto \wp(U); <e>(B) = \{g : (\exists h\in B)\, (g, h)\in R \} \\
[i]: \wp(U ) \longmapsto\wp(P ); [i](A) = \{h : (\forall g \in U)((g, h) \in R\longrightarrow g \in A)\}\\
[e] : \wp(P ) \longmapsto\wp(U); [e](B) = \{g : (\forall h\in P)((g, h)\in R\longrightarrow h\in B)\}
\end{align}
\end{definition}

It is known that the basic constructors may correspond to approximations under some conditions \cite{ppm2}. It may hold under some other conditions. Property system are not suitable for handling granularity and many of the inverse problem contexts. The latter part of the statement requires some explanation because suitability depends on the way in which the problem is posed - this has not been looked into comprehensively in the literature.

If all the data is of the form 
\[\text{Object } X \text{ is definitely approximated by } \{A_1, \ldots A_n \},\]
with the symbols $X,\, A_i$ being potentially substitutable by objects, then the data could in principle be written in property system form with the sets of $A_i$s forming the set of properties $P$ - in this situation the relation $R$ attains a different meaning. This is consistent with the structure being not committed to tractability of properties possessed by objects. Granularity would also be obscure in the situation. 

If all the data is of the form 
\[\text{Object } X \text{ 's approximations are included in } \{A_1, \ldots A_n \},\]
then the property system approach comes under even more difficulties. Granular operator spaces and generalised versions thereof \cite{am6900} in contrast can handle all this.

\section{Dialectical Negation}

In general rough sets, most relevant concepts of negation and implication are dialectical in some sense. A universal definition of \emph{dialectical negation} is naturally of interest - at least one that works for the associated vague contexts. Since vagueness is everywhere, multiple concepts of \emph{dialectical negation} in the literature need to be reconciled (to the extent possible) for the purpose of a universal definition. The main questions relating to possible definitions of dialectical negations or contradictions at the formal level arise from the following reasons (these are explained below):
\begin{itemize}
\item {The consensus that dialectical logics must be logics that concern state transitions.}
\item {The view that paraconsistent logics are essentially dialectical logics (see \cite{dw}). }
\item {The view that dialectical negation cannot be reduced to classical negation (see \cite{ip}). Indeed, in rough sets many kinds of negations and partial negations have been used in the literature (see for example, \cite{am501,pp1998,pp2000,am3,cc,bc1}, \cite{am105,am99,am24,sw,ccd11,pp20,gcd2018} and these lead to many contradictions as in 
\begin{enumerate}
\item {\emph{contradictions} \cite{pp1998} which are not false but that represent topological boundaries;}
\item {\emph{contradictions} \cite{pp2000} which are not false but lie between an absolute and local false;}
\item {\emph{contradictions} \cite{pp20} which lead to at least a paraconsistent and a paracomplete logic.}
\end{enumerate}}
\item {The view that dialectical negation is glutty negation (example \cite{ef14,br2008})- an intermediate kind of negation. }
\item {The view that only propositional dialectical logics are possible (\cite{mvp}).}
\item {The present author's position that dialectical contradiction must be represented by binary predicates or binary logical connectives in general \cite{am999,am699}. This is arrived at in what follows.}
\end{itemize}

The relationship of an object and its negation may belong to one of three categories (an extension of the classification in \cite{gp2006}):
\begin{enumerate}
\item {\textsf{Cancellation} as in \emph{the attributes do not apply}. The ethical category of \emph{the negative} as used in natural language is often about this kind of cancellation. It is easy to capture this in logics admitting different types of atomic variables (or formalised for instance by labeled deductive systems \cite{dg96}). In these the concept of the \textsf{Negative} is usually not an atomic category. Obviously this type of negation carries more information in being a \textsf{Not This but Among Those} kind of negation as opposed to the simple \textsf{Not This}, \textsf{Not This and Something Else } and weakenings thereof.}
\item {\textsf{Complementation} understood in the sense of classical logic.}
\item {\textsf{Glutty Version} understood as something intermediate between the above two.}
\end{enumerate}

In general rough sets, if $A$ is a subset of attributes and $c$ is set complementation, then the value of more common \emph{negations in a rough sense} include $A^{uc},\, A^{cl},\,A^{uuc},\, A^{clu} $. Each of these is a nonempty subset of $A^c$ in general. Consequently, the corresponding negation is glutty in a set theoretic sense. 

The concept of \emph{dialectical negation} as a material negation in logics about states is a reasonable abstraction of the core concept in Hegelian and Marxist dialectics (though this involves rejection of Hegel's idealist position). The negation refers to concepts in flux and so a logic with regard to the behavior of states rather than static objects would be appropriate. According to Hegel, the world, thought and human reasoning are dynamic and even the idea of true concepts are dynamic in nature. Poorly understood concepts undergo refinement as plural concepts (with the parts being \emph{abstract} in Hegel's sense) that assume many \emph{H(Hegelian)-Contradictory} forms. After successive refinements the resulting forms become reconciled or united as a whole. So for Hegel, H-contradictions are essential for all life and world dynamics. But Hegel's idealist position permitted only a closed world scheme of things. In Marx's materialist dialectic, the world is an open-ended system and so recursive applications of dialectical contradiction need not terminate or be periodic. All this means that the glutty interpretation of Hegel's contradiction may well be correct modulo some properties, while Marx's idea of dialectical contradiction does not reduce to such an interpretation in general. The debate on endurantism and perdurantism is very relevant in the context of Hegelian dialectics because the semantic domain associated is restricted by Hegel's world view. \emph{In rough semantics, especially granular ones, approximations may be seen as transitions and so the preconditions are met in a sense}.

\begin{example}
The identity of an apple on a table can be specified in a number of ways. For the general class of apples, a set of properties $X$ can be associated. The specific apple in question would also be possessing a set of specific properties $Z$ that include the distribution and intensity of colors. Obviously, many of the specific properties will not be true of apples in general. Further they would be in dialectical opposition to the general. Note that an agent can have multiple views of what $Z$ and $X$ ought to be and multiple agents would only contribute to the pluralism. In Hegel's perspective all of this dialectical contradiction must necessarily be resolved in due course (this process may potentially involve non materialist assumptions), while in the Marxist perspective a refined plural that may get resolved would be the result. Thus, the apple may be of the Alice variety of \emph{Malus Domestica} (a domestic apple) with many other specific features. The schematics (for an agent) is illustrated in \textsf{Fig. 1} - $\beth$ is a binary predicate with the intended meaning of being \emph{in dialectical opposition}.

To see how ideas of unary operations as dialectical negations can fail, consider the color of apples alone. If the collection of all possible colors of apples is known, then the set of all possible colors would be knowable. Negation of a white apple may be definable by complementation in this case. If on the other hand only a few possible sets of colors and some collections of colors are known, then the operational definition can fail (or lack in meaning). This justifies the use of a general binary predicate $\beth$ for expressing dialectical contradiction.

\begin{figure}[hbt]
\begin{center}
\begin{tikzpicture}[node distance=2.4cm, auto]
\node (A) {{\includegraphics{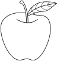}}};
\node (B) [right of=A]{$\mathbf{X}_1$};
\node (C) [above of=B] {$\mathbf{Z}_1$};
\node (H) [below of=A] {$\mathbf{X}_2$};
\node (K) [right of=H] {$\mathbf{Z}_2$};
\draw[->,font=\scriptsize,thick] (A) to node {General}(B);
\draw[->,font=\scriptsize,thick] (A) to node {Specific}(C);
\draw[->,font=\scriptsize,thick] (A) to node {General}(H);
\draw[->,font=\scriptsize,thick] (A) to node {Specific}(K);
\draw[-,dashed] (B) to node {$\beth$}(C);
\draw[-,bend right=-90, dashed] (H) to node {$\beth$}(C);
\draw[-,dashed] (H) to node {$\beth$}(K);
\end{tikzpicture}
\caption{Schematics of Dialectics of Identity}
\end{center}
\end{figure}
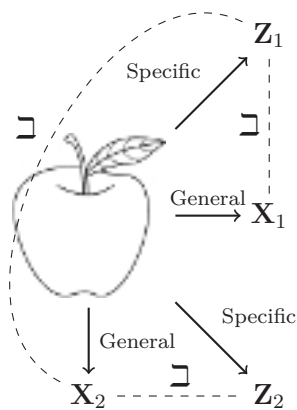 
 \end{example}

Da Costa et. al consider the heuristics of a possible dialectic logic in \cite{dw}. They seem to accept McGill's interpretation of unity of opposites \cite{mvp} and restrict themselves to a propositional perspective on the basis of difficulties with formalizing four of the six principles. This results in a very weak dialectical logic. They are however of the view that formal logics based on Marxist and Hegelian dialectics intersect the class of paraconsistent logics and there is great scope for deviation and that it can be argued that paraconsistent logics represent a desired amendment of  dialectics because of the latter's openness and non-rigid formalism.

The distinction between static and dynamic dialectical logics within the class of dialectical logics with dialectical contradiction as expressed with the help of an unary operation, may be attributed to Batens \cite{bd1,bd2}. Adaptive logics in that perspective would appear to be more general than dynamic dialectical logics; the main idea is to interpret inconsistencies as consistently as is possible. Key to this class of logics is the concept of tolerance of contradictory statements that are not necessarily reduced in their level of contradictions by way of proof mechanisms. Through this one can capture parts of the thesis-antithesis-synthesis meta principle. All semantic aspects of adaptive logics are intended in a classical perspective as opposed to dialethic logics and these are very closely connected to paraconsistent logics as well. Some Hegelian approaches (see \cite{al} for example) also seek to resolve universal contradictions.

Two of the most common misinterpretations or reductions of the concept of \emph{dialectical opposition} relate to excluding the very possibility of formalizing it and the reduction of dialectical negation to simple negation or opposites. Examples of the latter type are considered in \cite{hw,gj}. Some are of the view that Marx worked with normative ideas of concept and so introduction of related ideas in logic are improper. In modern terminology, Marx merely wanted concepts to be grounded in the material and was opposed to idealist positions that were designed for supporting power structures of oppression. This is reflected in Marx's position on Hegel's idealism and also, for example, on Wagner's position \cite{mecw24}. Marx and Hegel did not write about formal logic and the normative ideas of \emph{concept} used by both and other authors during the time can be found in great numbers. From the point of view of less normative (or non normative) positions  all these authors implicitly developed concepts at all times. It is also a reasonable position that Marxist methodologies should not be formalised independently of the normative restriction on possible ideas of concept afforded by actualization contexts. This is because it is always a good idea to have good grounding axioms for concepts to the extent that is permitted/possible by the context in question.   

In the present author's opinion for a methodology or theory to qualify as \emph{dialectic in Marx's sense} it is necessary that the idea of concepts used should be well grounded in the actualization contexts of the methodology or theory. In the context of reasoning with vagueness and rough sets, this may amount to requiring the approach be granular under specific conditions. 

In general, formal versions of dialectical logics can be based on some of the following principles/heuristics.
\begin{description}
\item[A]{Binary Logical predicates (that admit of universal substitution by propositional variables and well formed formulas) that are intended to signify binary dialectical contradiction are necessary for dialectical logics, }
\item[B]{Unary logical connectives (that admit of universal substitution by propositional variables and well formed formulas) suffice for expressing dialectical contradiction, }
\item[C]{The thesis-antithesis-synthesis principles must necessarily be included in the form of rules in any dialectical logic,}
\item[E]{Higher order quantifiers must be used in the logical formalism in an essential way because dialectical contradictions happen between higher order concepts,}
\item[F]{Dialectical logics should be realizable as fragments of first order predicate logic - this view is typically related to the position that higher order logics are superfluous.}
\item[G]{Dialectical contradiction in whatever representation must be present at each stage of what is defined to constitute dialectical derivation - this abstraction is due to the present author and is not hard to realise.}
\item[H]{All dialectic negations should be dialethic(*) in nature - this is a possibility explored in \cite{gp2006}. Dialethias are statements that can be both true and false simultaneously}
\item[I]{A logic that permits expression of progression of knowledge states is a dialectical logic.}
\item[J]{a first order logic perspective}
\item[K]{the point of view that dialectical contradiction can be expressed by binary predicates and not by unary operations}
\item[L]{Dialectical logics as paraconsistent logics incorporating contradictions or as inconsistency adaptive logics.}
\end{description}

Obviously many of these are mutually incompatible. [K] in particular is incompatible with [B] in the above. [I] leads to linear logics and variants. The meaning of dialectical logics that admit representation as a fragment of first order predicate logic will be naturally restricted. A few versions are known. Dynamic dialectical logics have been developed as inconsistency adaptive logics by Batens \cite{bd} in particular. In the present paper, \textsf{[A]} will be preferred as the binary predicate/predication cannot always be reduced to unary negations.

In general, it is obvious that given a dynamically changing subject, there will be at least a set of things which are dialectically contradictory to it in many ways. If $a$ is in dialectical contradiction to $b$ and $c$ in two different senses, then it is perfectly possible that $b$ is dialectically contradictory to $c$ in some other sense. Further if $X$ is dialectically contradictory to a conjunction of the form $Y\, \wedge\,Z$, then it is possible that $X$ is dialectically contradictory to $Y$ in some other sense and is virtually indifferent to $Z$.

\subsection {Dialectical Contradiction and Contradiction}

At a more philosophical level, the arguments of this section can be expressed in the language of functors \cite{ip}. However, a set-theoretic semantic approach is better suited for the present purposes. The concepts of contradiction and dialethic contradiction make essential use of negation operations (in some general sense), while that of dialectical contradiction when formulated on comparable terms does not necessarily require one. It is necessary to clarify the admissible negations in all this as many variations of the concept of logical negation are known and most are relevant for rough set contexts. 

Let $S$ be a partially ordered set with at least one extra partial unary operation $f$, a least element $\bot$ and a partial order $\leq$ ($\wedge, \, \vee$  being partial lattice infimum and supremum). In a partial algebra, two terms $p$ and $q$, are weakly equal, ($p\stackrel{\omega}{=}q$), if and only if \emph{whenever both terms are defined then they are equal}.
 Consider the following :
\begin{align}
\tag{N1}{ x\,\wedge\,f(x)\,\stackrel{\omega}{=}\,\bot  }\\
\tag{N2}{ (x\,\leq\,y\,\longrightarrow\,f(y)\,\leq\,f(x)) }\\
\tag{N3}{ x\,\leq\,f^{2}(x) }\\
\tag{N4}{ (x\,\leq\,f(y)\,\longrightarrow\,y\,\leq\,f(x)) }\\
\tag{N5}{ f^{n}(x)\,=\,f^{n+m}(x)  \text{ for some minimal } n,m\,\in\,N }\\
\tag{N6}{ f(x\,\vee\,y)\,\stackrel{\omega}{=}\,f(x)\,\wedge\,f(y) } \\
\tag{N9}{ x\,\wedge\,y\,=\,\bot\, \leftrightarrow\,y\,\leq\,f(x). } 
\end{align}

Distinct combinations of these non-equivalent conditions can characterise negations.
In \cite{at}, if an operation satisfies $N1$ and $N2$ over a distributive lattice, then it is taken to be a general negation. This is a reasonable concept for logics dealing with exact information alone as $N1$ does not hold in the algebras of vague logics, uncertain or approximate reasoning. For example it fails in the algebras of rough logic and generalizations thereof \cite{bc1,am24,am240,am9501}.

If $\forall{x}\,f^{m}(x)\,=\,f^{n}(x)$ holds, then the least $n$ such that $m\,<\,n$ is called the \emph{global period} of $f$,  $s\,=\,n\,-\,m$, the \emph{global pace} of $f$ and $(m,\,n)$, the index of $f$.
 
\begin{theorem}
When the poset is a distributive lattice in the above context, then the following are separately true:

\begin{enumerate}
\item {If $f$ satisfies $N1$ and $N2$, then the index $(0,\,n)$ for $n\,>\,2$ is not possible.}
\item {If $f$ satisfies $N1$, $N2$, and $N3$, then $f(\bot)\,=\,T$ is the greatest element of the lattice and $f(T)\,=\,\bot$.}
\item {$N1,\,N2,\,N3$ together do not imply $N9$}
\item {$N9$ implies each of $N1$, $N2$ and $N3$.}
\item {An interior operator $i$ on a poset is one that satisfies 
\begin{itemize}
\item {$i(x) \leq x$,}
\item {$(a \leq b \longrightarrow i(a) \leq i(b))$}
\item {$i(i(x))= i(x)$.}
\end{itemize}
If $f$ is a regular negation (that is it satisfies the conditions $N1$, $N2$ and $N3$) and $i$ an interior operator, then $g\,=\,if$ is a negation that satisfies $g^{4}\,=\,g^{2}$}
\end{enumerate}
\end{theorem}

Even at these levels of generality, the generalised negations can fail to express the appropriate concept of dialectical contradiction. 

In a set-theoretical perspective, if a set of things $A$ is dialectically opposed to another set $B$, then it may appear reasonable to expect $A$ to contain the set of things dialectically opposed to $B$. Subject to this intent, the set of \emph{all} things dialectically opposed to $A$ would need to be expressed by $\sim A$. But in dialectical reasoning it will still be reasonable to say that $A$ is dialectically opposed to some part of $\sim A$. For this the use of a unary $\sim$ can be glaringly insufficient. This is true not only from an algebraic system point of view (when working within a model) but also from perspectives generated by admissible sets of models. $N2$ is inherently incompatible with accepting $f$ as a unary dialectical negation operator, especially when a lattice order is expected to be induced by some concept of logical equivalence from the order. $N3$ is perhaps the most necessary property of a dialectical negation operation.

Well-formed formulas of certain derived types alone may admit of a negation (in the sense of being equivalent to one of the negatives). Such a negation is partial. For instance, $\sim\,\sim\,x$ may not be defined in the first place, and some of $N1-N9$ may hold for such negations. Using such types of negation for expressing dialectical contradictions through compound constructions may be possible in adequately labeled deductive systems.

Dialethic logics are logics that tolerate contradictions and accept true contradictions. To be more specific a dialetheia is a pair of statements having the form $A \& \neg A$ with each of $A$ and $\neg A$ being true. These statements may be interpreted sentences expressed in some language (that may be natural language or a language of thought, or anything). They can be used to formalize only some restricted cases of dialectical reasoning in which a unary dialectical contradiction operation is possible. It is also possible to reformulate some dialectical contexts as a dialethic process. Priest \cite{gp2} had indicated the possibility of using dialethic logics as a base for dialectic logics. In \cite{gp5}, Priest develops a dialectical tense logic, where it is possible for a system to exist in both pre and post states during (at the instant of) state transitions. Zeleny \cite{zj1} in particular has correctly pointed out (from a philosophical perspective) the possible shortcomings of a unary negation based approach. Though the issue of desiring incompatibility between classical logic and a possible dialectical logic is not a justified heuristic. The essential dialetheism principle is however usable in dialectical derivations. Such situations would allow dialectical opposition between proof patterns naturally. 

The nature and meaning of negation in a dialethic logic is explained in \cite{gp1,gp2}. From a philosophical meta perspective the negation of a formula is possibly a collection of formulas that may be represented by a single formula (from a logical perspective). It is with respect to such a \emph{negation} that dialethic logics must tolerate contradictions and accept true ones. A survey of concepts of contradictions for dialetheism can also be found in \cite{gp}. Using any kind of universal paraconsistent system for describing inconsistencies is virtually shown to be an undesirable approach in \cite{bd2}.  Marxist dialectics is perceived from a dialetheistic perspective of things in \cite{wj}. It is claimed that dialetheias correspond to and realise the concept of historical contradiction. The methodological aspect of Marxist dialectics is also ignored by Priest (see \cite{gp1990}) to the point that \emph{dialectic is a dialetheia}. There are no methodological strictures associated with dialetheias except for the requirement that they be real. This approach ignores 
\begin{itemize}
\item {the world view associated with Hegelian-Marxist dialectics,}
\item {the principle of unity of opposites, and}
\item {the basic problem with formalizing dialectical opposition with a unary operator - this is because the negation in dialectics is transient and dependent on the above two points.}
\end{itemize}

In the present author's view dialetheias do exist in the real world and they may be the result of 
\begin{itemize}
\item {missing data/information}
\item {a deliberate disregard for consistency. For example, a large number of people in the news media, religion and politics practice dialethic expression of a crude form and deceit. They may have their motivations for such actions, but those would not be justification for their dialetheias. It may also be possible to construct equivalent models or models with additional information that do not have true contradictions. 
For example, the statement \emph{$X$ has property $Q$ and $\neg Q$} may be replaced by the statements \emph{$X$ has property $Q$ in state $A$} and \emph{$X$ has property $\neg Q$ in state $B$}. This amounts to interfering with the data and does not really change the state of affairs. Many religious functionaries have been convicted of sexual crimes and most were in harmony with their apparent dialethic behavior (religious texts may be full of contradictions, that only allows for prolonging the derivations). }
\end{itemize}

The present author agrees with Priest's claims about dialetheias being not resolvable by revision of concepts and that they are better handled as such \cite{gp2006,gp1990,gp2014}. However, she does not agree at all that the proper way of formalizing Hegelian-Marxist dialectics in all cases is through dialetheias. A mathematical formulation of the issue for many sub cases may be possible through rough sets. The \emph{Cold vs Influenza} example considered in the subsection on examples of parthoods throws much light on the matter.

\subsection{Dialectical Predication}

At a philosophic level, dialectical predication is a relation between functors in the sense of \cite{ip}. At a formal model-theoretic semantic level, the best realization is through a binary dialectical predicate $\beth$, that may have limited connections with negation operations (if any). The basic properties that are necessary (not sufficient) of the predicate are as follows (with $\oplus$ standing for aggregation):

\begin{align}
 \beth(a, b) \longleftrightarrow \beth(b, a) \tag{Commutativity}\\
 \neg\, \beth (a, a) \tag{Anti-Reflexivity} \\
 \beth(a,b) \longrightarrow \beth(a\oplus c , b\oplus c) \tag{Aggregation}
\end{align}

This predicate may be related to unary dialectical negation operations in a simple or complex way. One possibility (that leads to $N4$) is the following (for two predicates $P$ and $Q$):
\[\beth (a, b) \text{ iff } P(a) \implies \neg Q(b) \]

That the definitions are important is illustrated by the following example.

Let $\{x_n \}$ be a sequence of real numbers. In contexts where reasoning proceeds from the concrete to a general, let 
\begin{itemize}
\item {\textsf{A} be the statement that \emph{Limit of the sequence is not conceivable}}
\item {\textsf{B} be the statement that \emph{Statement A is conceivable}}
\item {\textsf{C} be the statement that \emph{As B is true, the limit of the sequence is conceivable}.}
\end{itemize}

\textsf{A} is dialectically opposed to \textsf{C}, but the scenario does not amount to a dialetheia if the entire context has enough information of the process (of \textsf{B} being true) being referred to by \textsf{C}.

\begin{example}[Dialectics from Classification]

The intent of this example is to show that 
\begin{itemize}
\item {strategies for classification of information can be dialectically opposed to each other and}
\item {such information can fit into the rough set paradigm without the involvement of dialetheias.}
\end{itemize}

From an abstract perspective, consider a general process or phenomena $C$ described in abstract terms $A_1, \ldots , A_f$. Let every extension of the process have additional peculiar properties that lead to not necessarily independent classifications $C_1, C_2, \ldots , C_n$. Also let the categories $C^*_1, C^*_2, \ldots , C^*_r$ be formed by way of interaction between the members of said classes. This scenario leads to instances of parthood like $\pc A_1 C_1$ and $\pc C_1 C^*_1$ with $\pc C_1 C^*_1$ being in dialectical opposition to $\pc A_1 C_1$. Concrete instances of development over these lines are easy to find. In fact the historical development of any subject in the social sciences that has witnessed significant improvement over the last thirty years or so would fit under this schema. Two diverse contexts where such a dynamics may be envisioned are presented next.

Models and methods used for income estimation of rural agrarian households manage agrarian relations in different ways \cite{ms2015}. Long term ground level studies are required to clarify the nature of these relations. Suppose $C$ is about estimating poverty in a village and $A_1, A_2, A_3$ are abstract categories based on volume of monetary transaction by farmers. Some economists may use this for estimating net income and as an indicator of absence of poverty, while in reality farmers may be having negative income or the sources may not be reliable.

Suppose, an improved classification $C_1, C_2,$ $ C_3, C_4, C_5$ has been arrived at based on estimates of investment, expenditure, consumption, exchange of labor and other factors through ground level studies. The resultant classification may need to be improved further to take non monetary transactions like barter of goods and resources into account. Thus, $C_1^*, C_2^*, $ $C_3^*, C_4^*, C_5^*$  may be arrived at.

$C_1$ definitely takes $A_1$ into account and the latter is a causative factor for the former. This can be expressed by the parthood $\pc A_1 C_1$. $C_1$ is a much stronger causative factor of $C_1^*$. $\pc C_1 C_1^*$ then is dialectically opposed to $\pc A_1 C_1$. Such a relation can be used to track the context dynamics.

The subject of lesbian sexuality in particular has progressed significantly over the last few decades and can be expressed in similar abstract perspectives (see \cite{bz2000} and more recent literature). Women love women in different ways and this variation is significant for sub-classification. The parameters of classification relate to gender expression, sexual state variation, sexual performance, preferences in sexual interaction, routines, mutual communication, lifestyle choices, related social communication and more.   

\end{example}

\section{Dialectical Rough Sets}\label{cer}

A dialectical approach to rough sets was introduced and a more general program was formulated in \cite{am699} by the present author. Multiple kinds of roughly equivalent objects and the dialectical relation between them are stressed in the approach. This is reflected in the two algebraic logics proposed in the mentioned paper. The entire universe is taken to be a dialectical relation in the second semantics and possible derivations revolve around it. The main intent  was to include mixed kinds of objects in the semantics and so ideas of contamination apply differently. The essential content is repeated below (as \cite{am699} is a conference paper) and the nature of some possible parthoods involved is defined below.

A \emph{pre-rough algebra} \cite{bc1} is an algebra having the form \[S\,=\,\left\langle\underline{S}, \sqcap, \sqcup, \Rightarrow, L, \neg, 0, 1\right\rangle\] of type $(2,2,2,1,1,0,0)$, which satisfies:
\begin{itemize}
\item {$\left\langle\underline{S},\sqcap,\sqcup,\neg\right\rangle$ is a De Morgan lattice.}
\item {$\neg\neg{a}=a \,;\, L(a)\,\sqcap\,{a}=L(a)$}
\item {$ L(a\,\sqcup\, b)=L(a)\,\sqcup \,L(b)\,;\,\neg{L}\neg{L}(a)=L(a)$}
\item {$LL(a)=L(a)\, ;\,\ L(1)=1\, ;\, L(a\sqcap{b})=L(a)\sqcap L(b) \,;\, \neg{L}(a)\sqcup{L}(a)=1$}
\item {If $L(a)\sqcap{L(b)}=L(a)$ and $\neg{L(\neg(a\sqcap{b}))}=\neg{L}(\neg{a})$ then ${a\sqcap{b}=a}$. This is actually a quasi equation.}
\item {${a\Rightarrow{b}}=(\neg{L}(a)\sqcup{L}(b))\sqcap(L(\neg{a})\sqcup\neg{L}(\neg{b}))$.}
\end{itemize}
A completely distributive pre-rough algebra is called a \emph{rough algebra}. In all these algebras it is possible to define an operation $\diamond$ by setting $\diamond (x)\,=\,\neg{L}\neg(x)$ for each element $x$.

It should be noted that above definition has superfluous conditions. An equivalent definition that was used in \cite{am501} (based on \cite{ajm2015}) is the following: 

\begin{definition}
An \emph{essential pre-rough algebra} is an algebra of the form \[E = \left\langle\underline{E}, \sqcap, L, \neg, 0, 1  \right\rangle\]  that satisfies the following (with $\sqcup$ being a defined by $(\forall a, b)\, a\sqcup b = \neg (\neg a \sqcap \neg b) $ and $a \leq b$ being an abbreviation for $a\sqcap b = a$.)
\begin{description}
\item []{$\left\langle\underline{E}, \sqcap, \sqcup, \neg, 0, 1  \right\rangle$ is a quasi Boolean algebra.}
\item [E1]{$L1= 1$}
\item [E2]{$(\forall a)\, La \sqcap a = La $}
\item [E3]{$(\forall a, b)\,L(a\sqcap b) = L(a) \sqcap L(b)$}
\item [E4]{$(\forall a)\,\neg L \neg L a = La$}
\item [E5]{$(\forall a )\,\neg L a\sqcap La = 0$}
\item [E6]{$(\forall a, b)\, (\neg L \neg a \leq \neg L \neg b  \,\&\, La \leq Lb \longrightarrow a \leq b )$}
\end{description}
An essential pre-rough algebra is said to be an \emph{essential rough algebra} if $L(E)$ is also complete and completely distributive- that is it satisfies (for any subset $X$ and element $a$) \[a \sqcup (\bigsqcap X) = \bigsqcap \{a\sqcup x : x\in X\} \, \&\, a \sqcap (\bigsqcup X) = \bigsqcup \{a\sqcap x : x\in X\}  \]  
\end{definition}

Essential pre-rough algebras are categorically equivalent to pre-rough algebras and essential rough algebras to rough algebras.

In this semantics explicit interaction between objects in the rough semantic domain and entities in the classical semantic domain is permitted. The requirement of explicit interaction is naturally tied to objects having a dual nature in the relatively \emph{hybrid} semantic domain. Consequently an object's existence has dialectical associations. In application contexts, this approach can also be useful in enriching the interaction within the rough semantic domain with additional permissible information from the classical semantic domain. 

Suppose $S_1,\, S_2, \, S_3,$ and $S_4$ are four general approximation spaces. Suppose that the semantics of $S_2$ relative to $S_1$ and $S_4$ relative to $S_3$  are definable in a semantic domain. The question of equivalence of these \emph{relativizations} is relevant. It may be noted that the essential problem is implicit in \cite{am240}. The hybrid dialectical approach is relevant in these contexts. But of course, this approach is not intended to be compatible with contamination.

In \emph{rough algebra semantics} it is not possible to keep track of the evolution of rough objects relative to the classical semantics suggested by the Boolean algebra with approximation operators. Conversely in the latter it is not possible to form rough unions and rough intersections relative to the \emph{rough algebra} semantics. These are examples of relative distortions. The CERA semantics (\emph{concrete enriched pre-rough algebra}), which is developed in the next subsection can deal with this, but distortions relative to \emph{super rough semantics} (\cite{am3}) are better dealt with CRAD (\emph{concrete rough dialectical algebra} introduced in the last subsection) like methods. For more on these considerations, the reader is referred to \cite{am240} and in the \emph{three-valued perspective} to \cite{pp2000,pp1998}.  In \cite{pp2000}, a three-valued sub domain and a classic sub domain formed by the union of the singleton granules of the classification are identified within the  rough domain itself.

Jaskowski's discursive logic is an example of a subvaluationary approach in that it retains the non truth of contradictions in the face of truth-gluts. Connections with pre-rough, rough algebras and rough logics are well known (see \cite{bcb}). In particular, Pawlak's five valued rough logic $R_{l}$ (see \cite{zp0}) and $L_{R}$ (\cite{bcb}) are not dialethic: though it is possible to know that something is roughly true and roughly false at the same time, it is taken to be \emph{roughly inconsistent} as opposed to being just \emph{true} or \emph{roughly true}. This rejection definitely leads to rejection of other reasoning that leads to it as conclusion.  Importantly a large set of logics intended for capturing rough semantics are paraconsistent and make use of skewed forms of conjunction and disjunction. It can be argued that the latter feature is suggestive of incomplete development of the logics due to inconsistencies in the application of the underlying philosophy (see \cite{hc} for example). The 4-valued DDT (see \cite{at}) addresses some of these concerns with a justification for 3-valuedness in some semantics of classical RST. The NMatrix based logic (\cite{ak}) provides a different solution by actually avoiding conjunction and disjunction operations (it should, however, be noted that conjunctions and disjunctions are definable operations in the NMatrix based logic). In super rough semantics, due to the present author \cite{am3}, the ability of objects to approximate is called in. These concerns become more acute in the semantics of more general rough sets. 

In summary, the main motivations for the approach of this section are
\begin{itemize}
\item {to provide a framework for investigating relative distortions introduced by different theories - this is in sharp contrast to the contamination reduction approach \cite{am240,am3} of the present author,}
\item { to improve the interface between rough and classical semantic domains in application contexts,}
\item {to investigate \emph{relativization of semantics} in the multi source general rough contexts (or equivalently in the general dynamic approximation contexts) - in \cite{am909} a distinct semantic approach to the problem is developed by the present author,}
\item {address issues relating to truth and parthood at the semantic level,}
\item {and develop a  dialectical logic of rough semantics.}
\end{itemize}

The nature of parthood was not considered in the context by the present author at the time of writing \cite{am699}.  It is considered here to specify the nature of dialectical oppositions and potential diagrams of opposition.  

\subsection{Enriched Classical Rough Set Theory}

Let $S\,=\,\left\langle \underline{S},\,R\right\rangle $ be an approximation space with $\underline{S}$ being a set and $R$ an equivalence. $S$ will be used interchangeably with $\underline{S}$ and the intended meaning should be clear from the context. If $A\,\subset S$, $A^{l}\,=\,\bigcup\{[x]\,;\,[x]\,\subseteq\,A\}$ and $A^{u}\,=\,\bigcup\{[x]\,;\,[x]\,\cap\,A\,\neq\,\emptyset\}$ are the \emph{lower} and \emph{upper approximation} of $A$ respectively. If $A,\,B\,\in\,\wp (S)$, then $A$ is \emph{roughly equal} to $B$ ($A\,\approx\,B$) if and only if $A^{l}\,=\,B^{l}$ and $A^{u}\,=\,B^{u}$. $[A]$ shall be the equivalence class (with respect to $\approx$) formed by a $A\,\in\,\wp (S)$. 

The proposed model may be seen as an extension of the pre-rough and rough algebra models in \cite{bc1}. Here the  base set is taken to be $\wp (S)\,\cup\,\wp (S)|\approx$ as opposed to $\wp (S)|\approx$ (used in the construction of a rough set algebra). The new operations $\oplus$, and $\odot$ introduced below are correspond to generalised aggregation and commonality respectively in the mixed domain. This is not possible in classical rough sets proper.

\begin{definition}
On $Y\,=\,\wp (S)\,\cup\,\wp (S)|\approx$, the operations $\ml,\,\oplus,\,\odot,\,\mm,\,\rightsquigarrow,\,\twoheadrightarrow,\,\sim$ are defined as follows: (it is assumed that the operations $\cup,\,\cap,\,^{c},\,^{l},\,^{u}$ and $\sqcup,\,\sqcap,\,L,\,M,\,\neg,\,\Rightarrow $ are available on $\wp (S)$ and $\wp(S)|\approx$ respectively. Further \[\tau_{1} x \Leftrightarrow x\in \wp(S) \text{ and } \tau_{2}x \Leftrightarrow x\in \wp(S)|\approx .\]
\begin{itemize}
\item {
\[\ml x = \left\lbrace  \begin{array}{ll}
 x^{l} & \text{if } \tau_{1}x\\
 Lx & \text{if } \tau_{2}x\\
 \end{array} \right.\]
}
\item { 
\[\blacklozenge x = \left\{ \begin{array}{ll}
 x^{u} & \text{ if } \tau_{1}x \\
 \neg L \neg x & \text{ if } \tau_{2}x \\
 \end{array} \right. \]
 }
\item { 
\[ x \oplus y = \left\{ \begin{array}{ll}
 x\,\cup\,y & \text{ if } \tau_{1}x,\,\tau_{1}y \\
\left[ x\,\cup\,(\bigcup_{z\,\in\,y} z)\right]  & \text{ if } \tau_{1}x,\,\tau_{2}y \\
\left[ (\bigcup_{z\,\in\,x} z)\,\cup\,y\right]  & \text{ if } \tau_{2}x,\,\tau_{1}y  \\
 x\,\sqcup\,y & \text{ if } \tau_{2}x,\,\tau_{2}y 
 \end{array} \right.\]
 } 
\item { 
\[ x \odot y = \left\{ \begin{array}{ll}
 x\,\cap\,y & \text{ if } \tau_{1}x,\,\tau_{1}y \\
 \left[ x\,\cap\,(\bigcap_{z\,\in\,y} z)\right]  & \text{ if } \tau_{1}x,\,\tau_{2}y \\
 \left[ (\bigcap_{z\,\in\,x}\,z)\,\cap\,y\right]  & \text{ if } \tau_{2}x,\,\tau_{1}y  \\
 x\,\sqcap\,y & \text{ if } \tau_{2}x,\,\tau_{2}y 
 \end{array} \right.\]
 }
\item { 
\[ \sim x\, = \left\{ \begin{array}{ll}
 x^{c} & \text{ if } \tau_{1}x \\
 \neg x & \text{ if } \tau_{2}x 
 \end{array} \right. \]
 }
\item { 
\[ x \rightsquigarrow y = \left\{ \begin{array}{ll}
 x\,\cup\,y^{c} & \text{ if } \tau_{1}x,\,\tau_{1}y \\
 \left[\bigcup_{z\,\in\,y} (x\,\cup\, z^{c})\right] & \text{ if } \tau_{1}x,\,\tau_{2}y \\
 x\implies y & \text{ if } \tau_{2}x,\,\tau_{2}y \\
 \left[\bigcup_{z\,\in\,x}(z\,\cup\,y^{c})\right] & \text{ if } \tau_{2}x,\,\tau_{1}y 
  \end{array} \right.\]
 }
\item { 
 \[ x \twoheadrightarrow y = \left\{ \begin{array}{ll}
 [x\,\cup\,y^{c}] & \text{ if } \tau_{1}x,\,\tau_{1}y \\
 \left[\bigcup_{z\,\in\,y} (x\,\cup\, z^{c})\right] & \text{ if } \tau_{1}x,\,\tau_{2}y \\
 x\implies y & \text{ if } \tau_{2}x,\,\tau_{2}y \\
 \left[\bigcup_{z\,\in\,x}(z\,\cup\,y^{c})\right] & \text{ if } \tau_{2}x,\,\tau_{1}y 
 
 \end{array} \right.\]
 } 
\end{itemize}
\end{definition}

It should be noted that $\odot$ is a very restrictive operation (because the commonality is over a class) when one of the argument is of type $\tau_1$ and the other is of type $\tau_2$. An alternative is to replace it with $\circ$ defined by the Eqn.\ref{circ}.

\begin{equation}\label{circ}
x\,\circ a = \left\{ \begin{array}{ll}
 x\,\cap\,a & \text{ if } \tau_{1}x,\,\tau_{1}a \\
 \left[ x\,\cap\,(\bigcup_{z\,\in\,a} z)\right]  & \text{ if } \tau_{1}x,\,\tau_{2}a \\
 \left[ (\bigcup_{z\,\in\,x}\,z)\,\cap\,a\right]  & \text{ if } \tau_{2}x,\,\tau_{1}a  \\
 x\,\sqcap\,a & \text{ if } \tau_{2}x,\,\tau_{2}a 
 \end{array} \right. 
\end{equation}

\begin{definition}
In the above context a partial algebra of the form \[W=\left\langle \underline{\wp (S)\cup \wp(S)|\approx},\,\neg,\,\sim,\,\oplus,\,\odot,\,\mm,\,\ml,\,0,\,1,\,\bot,\,\top \right\rangle\] of type $(1,\,1,\,2,\,2,\,1,\,1,\,0,\,0,\,0,\,0)$ is a \emph{concrete enriched pre-rough algebra}\\ (CERA) if a pre-rough algebra structure is induced on $\wp(S)|\approx$. \emph{Concrete enriched rough algebras} can be defined in the same manner. If the approximation space is $X$, then the derived CERA will be denoted by $\mathfrak{W}(X)$. Note that the two implication-like operations are definable in terms of other basic functions. A CERA in which $\odot$ has been replaced by $\circ$ is said to be a \emph{soft} CERA.
\end{definition}

\begin{proposition}
CERAs are well defined because of the representation theory of pre-rough algebras. 
\end{proposition}

\begin{theorem}
A CERA satisfies all the following: (The first two conditions essentially state that the $\tau_{i}$s are abbreviations)
\begin{gather*}
 \tag{type-1 } (x \rightsquigarrow x = \top  \longleftrightarrow \tau_{1} x)  \\
 \tag{type-2 } (\neg x = \neg x  \longleftrightarrow \tau_{2} x )  \\
 \tag{ov-1 } \sim \sim x = x   ;\;   \ml \ml x = \ml x    ;\;  \blacklozenge \ml x = \ml x   \\
 \tag{ov-2 } \ml x \oplus x = x    ;\;   \ml x \odot x = \ml x    ;\;  \blacklozenge x \oplus x =\blacklozenge x    ;\;  \blacklozenge x \odot x = x  \\
 \tag{ov-3 } \ml\blacklozenge x =\blacklozenge x    ;\;   x  \oplus x = x    ;\;   x \odot x = x   \\
 \tag{qov-1 } (\tau_{1}x \longrightarrow \sim x \oplus   x = \top)   ;\;   (\tau_{2}x \longrightarrow \sim \ml x \oplus \ml x = 1)  \\
 \tag{qov-2 } \sim \bot = \top   ; \;  \sim 0 = 1  \\
 \tag{u1 } x \oplus (x \oplus (x \oplus y)) = x \oplus (x \oplus y)    ;\;   x \odot (x \odot (x \odot y)) = x \odot (x \odot y)   \\
 \tag{u2 } x \oplus y = y \oplus x  ;\;   x \odot y = y \odot x  \\
 \tag{ter(i1) } (\tau_{i}x, \tau_{i}y, \tau_{i}z \longrightarrow x \oplus (y \oplus z) = (x \oplus y) \oplus z) ; i=1,2\\
 \tag{ter(i2} (\tau_{i}x, \tau_{i}y, \tau_{i}z \longrightarrow x \oplus (y \odot z) = (x \oplus y)\odot (x \oplus z)) ; i=1,2 \\
 \tag{ter(i3} (\tau_{i}x, \tau_{i}y, \tau_{i}z \longrightarrow x \odot (y \odot z) = (x \odot y) \odot z)   ; i=1,2\\
 \tag{bi(i) } (\tau_{i}x, \tau_{i}y \longrightarrow x \oplus (x \odot y) = x,  \sim (x \odot y) = \sim x \oplus \sim y)   ; i=1,2 \\
 \tag{bm } (\tau_{1}x, \tau_{2}y, x \oplus y = y \longrightarrow\blacklozenge x \oplus y = y)  \\
 \tag{hra1 } (\tau_{1}x, (1 \odot x = y) \vee (y = x \oplus 0)\longrightarrow \tau_{2}y)   
\end{gather*}
\end{theorem}

\begin{definition}
An \emph{abstract enriched pre-rough partial algebra} (AERA) will be a partial algebra of the form \[S\,=\,\left\langle\underline{S},\,\neg,\,\sim,\,\oplus,\,\odot,\,\mm,\,\ml,\,0,\,1,\,\bot,\,\top  \right\rangle\] (of type 
$(1,\,1,\,2,\,2,\,1,\,1,\,0,\,0,\,0,\,0)$) that satisfies:
\begin{description}
\item [RA]{$\mathrm{dom}(\neg)$ along with the operations $(\oplus,\,\odot,\,\mm,\,\ml,\,\sim,\,0,\,1)$ restricted to it and the definable $\Rightarrow$ forms a pre-rough algebra,}
\item [BA]{$\underline{S}\setminus \mathrm{dom}(\neg)$ with the operations $(\oplus,\,\odot,\,\mm,\,\ml,\,\sim,\,\top,\,\bot)$ restricted to it forms a topological boolean algebra (with an interior and closure operator),}
\item [IN]{Given the definitions of type-1, type-2, all of u1, u2, ter(ij), bi(i), bm and hra hold for any $i, \, j$.}
\end{description}
Note that AERAs are actually defined by a set of quasi equations.
\end{definition}

\begin{theorem}
Every AERA $S$ has an associated approximation space $X$ (up to isomorphism), such that the derived CERA $\mathfrak{W}(X)$ is isomorphic to it. 
\end{theorem}
\begin{proof}
Given $S$, the topological boolean algebra and the pre-rough algebra part can be isolated as the types can be determined with $\neg,\,\oplus$ and the $0$-place operations. The representation theorems for the parts can be found in \cite{rh} and \cite{bc1} respectively.

Suppose $\mathfrak{W}(Y)$ is a CERA formed from the approximation space $Y$ (say) determined by the two parts. If $Y$ is not isomorphic to $X$ as a relational structure, then it is possible to derive an easy contradiction to the representation theorem of the parts.

Suppose $\mathfrak{W}(X)$ is not isomorphic to $S$, then given the isomorphisms between respective parts, at least one instance of $x\oplus ' y\,\neq\, x\oplus y$ or $x\odot ' y\,\neq\,x\odot y$ (for a type-1 $x$ and a type-2 $y$ with $'$ denoting the interpretation in $\mathfrak{W}(X)$). But as type-1 elements can be mapped into type-2 elements (using $0$ and $\oplus$), this will result in a contradiction to the representation theorem of parts. 

\end{proof}

\subsection{Dialectical Rough Logic}

A natural dialectical interpretation can be assigned to the proposed semantics. A subset of the original approximation space has a dual interpretation in the classical and rough semantic domain. While an object in the latter relates to a set of objects in the classical semantic domain, it is not possible to transform objects in the rough domain to the former. For this reason, the universe is taken to be the set of tuples having the form $\{(x,\,0 \oplus x)\,:\,\tau_{1}x\}\,\cup\,\{(b,\,x)\,:\,\tau_{2} b \,\&\,\tau_{1} x \,\&\,x \oplus 0 = b\}\,=\,K$ ($x$ and $b$ being elements of a CERA). This universe is simply the described dialectical relation between objects in the two domains mentioned above. Other dialectical relations can also be derived from the specified one.

\begin{definition} A \emph{concrete rough dialectical algebra} (CRAD) will be a partial algebra on $K$ along with the operations $+,\,\cdot,\,\ml^{*},\,\neg,\,\thicksim$ and $0$-place operations $(\top,\,1),\,(1,\,\top),\,(0,\,\bot),\,(\bot,\,0)$ defined by (EUD is an abbreviation for \textsf{Else Undefined})
\[(a,b)\,+\,(c,e) = \left\{ \begin{array}{ll}
 (a\oplus c,\,b\oplus e) & \text{ if }\tau_{i}a,\,\tau_{i}c \text{ if defined} \\
 (a\oplus c,\,e\oplus a )  & \text{ if }\tau_{1} a,\,\tau_{2} c ,  (e\oplus a)\oplus 0\,=\,a \oplus c , \text{ EUD}\\
 (a\oplus e,\,c\oplus b )  & \text{ if }\tau_{2} a,\,\tau_{1} c ,  (c\oplus b)\oplus 0\,=\,a \oplus e , \text{ EUD}
 \end{array} \right.\]
\[(a,b)\,\cdot\,(c,e) = \left\{ \begin{array}{ll}
 (a\odot c,\,b\odot e) & \text{ if }\tau_{i}a,\,\tau_{i}c \text{ if defined}\\
 (a\odot c,\,e\odot a )  & \text{ if }\tau_{1} a,\,\tau_{2} c ,  (e\odot a)\odot 0\,=\,a \odot c ,\text{ EUD}\\
 (a\odot e,\,c\odot b )  & \text{ if }\tau_{2} a,\,\tau_{1} c ,  (c\odot b)\odot 0\,=\,a \odot e , \text{ EUD}
 \end{array} \right.\]

$\ml^{*}(a,\,b)\,=\,(\ml a,\,\ml b)$ if defined and  $\thicksim (a,\,b)\,=\,(\sim a,\,\sim b)$ if defined.
\end{definition}

\subsection*{Illustrative Example}
 
The following example is intended to illustrate key aspects of the theory invented in this section. 
 
Let $S= \{a, b, c, e, f, q\} $ and $R$ be the least equivalence relation generated by 
\[\{(a, b ),\,(b, c),\,(e, f)\}. \]
Under the conditions, the partition corresponding to the equivalence is 
\[\mathcal{G} = \{ \{a, b, c \},\,\{e, f \},\,\{q \}  \}.\]
The quotient $S|R$ is the same as $\mathcal{G}$. In this example strings having the form $ef$ are used as an abbreviation for $\{e, f\}$.

The set of triples having the form $(x, x^l, x^u)$ for any $x\in \wp(S)$ is as below:
\begin{itemize}
\item {$(a, \emptyset,abc)$,  $(b, \emptyset, abc)$,  $(c, \emptyset, abc)$,  $(e,\emptyset, ef)$,} 
\item {$(f, \emptyset, ef)$, $(q, q, q)$, $(ab, \emptyset,abc)$,   $(ac, \emptyset, abc)$,} 
\item {$(ae, \emptyset,abcef)$,  $(af, \emptyset, abcef)$,  $(aq, q, abcq)$, $(bc, \emptyset, abc)$,}
\item {$(be, \emptyset, abcef)$, $(bf, \emptyset, abcef)$, $(bq, q, abcq)$, $(ec, \emptyset, abcef)$,}
\item {$(cf, \emptyset, abcef)$, $(ef,ef, ef)$,  $(eq, q, efq)$,  $(fq, q, efq)$, $(abc, abc, abc)$,}
\item {$(abe,\emptyset, abcef)$, $(abf,\emptyset, abcef)$, $(abq, q, abcq)$, $(bce,\emptyset,abcef)$,}
\item {$(bcf,\emptyset,abcef)$, $(bcq,q,abcq)$, $(ace,\emptyset,abcef)$, $(acf,\emptyset,abcef)$,}
\item {$(acq,q,abcq)$, $(aef,ef,abcef)$, $(bef,ef,abcef)$, $(cef,ef,abcef)$,} 
\item {$(aeq,q,S)$, $(afq,q,S)$, $(beq,q,S)$, $(bfq,q,S)$, $(ceq,q,S)$,}
\item {$(cfq,q,S)$, $(efq,efq,efq)$, $(abce,abc,abcef)$, $(abcf,abc,abcef)$,}
\item {$(abcq,abcq,abcq)$, $(abef,ef,abcef)$, $(abeq,q,S)$, $(abfq,q,S)$,} 
\item {$(bcef,ef,abcef)$, $(bceq,q,S)$, $(bcfq,q,S)$, $(aceq,q, S)$, $(acfq,q, S)$,}
\item {$(acef,ef,abcef)$, $(aefq,efq,S)$, $(befq, efq,S)$, $(cefq, efq,S)$,} 
\item {$(abcef,abcef,abcef)$, $(abceq,abcq,S)$, $(abcfq,abcq,S)$, $(acefq,efq,S)$,}
\item {$(bcefq,efq, S)$, $(S, S, S)$}
\end{itemize}

 From the values, it can be checked that the sets of roughly equivalent objects are
 \begin{itemize}
\item {$\{a,b, c,ab,ac,bc\}$, $\{e, f\}$, $\{q\}$ - the reader should note that elements belonging to the set are themselves sets.}
\item {$\{ae,af,be, bf,ce,cf, abe, ace, acf,abf,bce,bcf \}$}
\item {$\{abq,acq,bcq,aq,bq,cq\}$, $\{abce,abcf\}$, $\{aef, bef, cef,abef, acef,bcef\}$}
\item {$\{eq, fq\}$, $\{abc\}$, $\{abcef\}$, $\{ef\}$, $\{abcq\}$, $\{efq\}$, $S$}
\item {$\{aeq,beq,ceq,afq,bfq,cfq,abeq,aceq,bceq,abfq,bcfq,acfq\}$}
\item {$\{aefq, befq, cefq, abefq, bcefq,acefq\} $ and $\{abceq,abcfq\}$.}
\end{itemize}

The domain is taken to be $\wp(S) \cup \wp(S)|R$ in case of a CERA and interpretations of the unary operations are obvious. The nontrivial binary operations get interpreted as below :
\[ bc\oplus[bf] = \left[ bc \cup abcef \right] = [abcef]\]
\[b\oplus [f] =  [bef] = \{aef, bef, cef,abef, acef,bcef\}\]
\[ bc \odot [bf] = 
 \left[ bc\,\cap\,\bigcap \{ae,af,be, bf,ce,cf, abe, ace, acf,abf,bce,bcf \}\right] = [\emptyset]\]
\[b\odot [f] = [b\cap \bigcap \{e, f\}] = [\emptyset] \] 
\[abcq \odot [q] = [q] \] 
\[ bc \rightsquigarrow [bf] = \left[\bigcup_{z\,\in\,[bf]} (bc\,\cup\, z^{c})\right] = S\]
\[bc \rightsquigarrow [abceq] = \{abcef\} \]
\[bc \rightsquigarrow [S] = \{a, b, c, ab, bc, ac \}\]
\[[bf] \twoheadrightarrow bc =  \left[\bigcup_{z\,\in\,[bf]} (z\,\cup\, bc^{c})\right] = [S] \]

The universe of the CRAD associated with a CERA $S$ is formed as the set of pairs having the form $(x,0\oplus x)$ and   $(x\oplus 0, x)$ under the restriction that $\tau_1 x$ holds. So in the present example, some elements belonging to the universe are $(a, \{a,b, c,ab,ac,bc\})$, $(\{a,b, c,ab,ac,bc\}, bc)$, $(fq,\{eq, fq\})$.

\[(a,b)\,+\,(c,e) = \left\{ \begin{array}{ll}
 (a\oplus c,\,b\oplus e) & \text{ if }\tau_{i}a,\,\tau_{i}c\,\ \\
 (a\oplus c,\,e\oplus a )  & \text{ if }\tau_{1} a,\,\tau_{2} c ,  (e\oplus a)\oplus 0\,=\,a \oplus c , \text{ EUD}\\
 (a\oplus e,\,c\oplus b )  & \text{ if }\tau_{2} a,\,\tau_{1} c ,  (c\oplus b)\oplus 0\,=\,a \oplus e , \text{ EUD}
 \end{array} \right.\]

 To compute $(a, \{a,b, c,ab,ac,bc\}) + (\{eq, fq\}, fq)$ it is necessary to compute 
 \begin{itemize}
\item {$a \oplus \{eq,fq\}$ = $\{aefq, befq, cefq, abefq, bcefq,acefq\} $}
\item {$ a \oplus fq$ = $afq$}
\item {$afq +0 = \{aeq,beq,ceq,afq,bfq,cfq,abeq,aceq,bceq,abfq,bcfq,acfq\}$ }
\item {Clearly, $a \oplus \{eq,fq\} \neq afq +0$.  }
\item {So $(a, \{a,b, c,ab,ac,bc\}) + (\{eq, fq\}, fq)$ is not defined.}
\end{itemize}

Also note that $(a, \{a,b, c,ab,ac,bc\}) + (b, \{a,b, c,ab,ac,bc\})$ is defined, but $(a, \{a,b, c,ab,ac,bc\}) + (bc, \{a,b, c,ab,ac,bc\})$ is not.
 
\paragraph{Dialectical Negations in Practice}
Real examples can be constructed (from the above example) by assigning meaning to the elements of $S$. An outline is provided below. 
\begin{itemize}
\item {Suppose $\{a, b, c, e, f, q\}$ is a set of attributes of lawn tennis players.}
\item {For the above sentence to fit into the example context, it is necessary that they can be freely collectivised. This means that no combination of attributes should be explicitly forbidden.}
\item {While sets of the form $ec$ refer to players with specific attributes, roughly equal objects like $\{aef, bef, cef,abef, acef,bcef\}$ can be read as new class labels. Members of a class can be referred to in multiple ways. }
\item {The operations defined permit aggregation, commonality and implications in novel ways. The $\oplus$ operation in particular can generate new classes that contain classes of roughly equal players and players with specific attributes in a mereological sense. For example, it can answer questions of the form: \emph{What features can be expected of those who have a great backhand and have roughly equal performance on hard courts?} }
\item {A number of dialectical negations can be defined in the situation. For example,
\begin{itemize}
\item {An object $x$ is in a sense dialectically opposed to a roughly equivalent set of objects $H$.}
\item {An object $x$ is in a sense dialectically opposed to $x \oplus H$.}
\item {Likewise other operations defined provide more examples of opposition.}
\end{itemize}
}
\end{itemize}

More generally, similar dialectical negation predicates can be defined over CRADs and new kinds of logical rules can be specified that concern transformation of one instance of dialectical negation into another, restricted introduction and inference rules. Related logic will appear separately.

\subsection{Parthoods}

In \textsf{CERA} related contexts, the universe is taken to be $W=\underline{\wp (S)\cup \wp(S)|\approx}$ and the most natural parthoods are ones defined from the aggregation and commonality operators. Parthoods can also be based on information content and ideas of \emph{consistent comparison}.

\begin{definition}
The following parthoods\label{cerpart} can be defined in the mixed semantic domain corresponding to \textsf{CERA} on $W$ 
\begin{align}
\pc_o a b \; \leftrightarrow \;  [a] \leq [b] \tag{Roughly Consistent}\\
\pc_\oplus a b \; \leftrightarrow \; a \oplus b = b  \tag{Additive}\\
\pc_\odot a b \;  \leftrightarrow \; a \odot b = a \tag{Common} 
\end{align}
\end{definition}

$\leq$ is the lattice order used in the definition of pre-rough algebras. Note that the operations $\oplus , \odot$ are not really required in the definitions of the last two parthoods which can equivalently be defined using the associated cases. This is significant as one of the goals is to \emph{count the objects in specialised ways to arrive at semantics that make sense} \cite{am240}.

In the definition of the base set $K$ of \textsf{CRAD}, $K$ is already a dialectic relation. Still definitions of parthoods over it make sufficient sense.

\begin{definition}
The relation $\pc_\aleph$, defined as below, will be called the \emph{natural parthood} relation  on  $K$:
\[\pc_\aleph a b \leftrightarrow [e_1 a]\leq [e_1 b ] \,\&\, [e_2 a ] \leq [e_2 b] ,\]
where the operation $e_i$ gives the $i$th component for $i=1,2$.
\end{definition}

Admittedly the above definition is not internal to $K$ as it refers to things that do not exist within $K$ at the object level of reasoning.

\section{General Parthood}

Parthood can be defined in various ways in the framework of rough sets in general and granular operator spaces in particular. The rough inclusion defined earlier in the background section is a common example of parthood. Some others have been introduced in Def.\ref{cerpart} and in the illustrative example. The following are more direct possibilities that refer to a single non classical semantic domain (the parthoods of CERA are in the classical domain):

\begin{align}
 \tag{Very Cautious} \pc ab \longleftrightarrow a^l \subseteq b^l \\
 \tag{Cautious} \pc ab \longleftrightarrow a^l \subseteq b^u \\
 \tag{Lateral} \pc ab \longleftrightarrow a^l \subseteq b^u\setminus b^l \\
 \tag{Possibilist} \pc ab \longleftrightarrow a^u \subseteq b^u \\
 \tag{Ultra Cautious} \pc ab \longleftrightarrow a^u \subseteq b^l \\
 \tag{Lateral+} \pc ab \longleftrightarrow a^u \subseteq b^u\setminus b^l \\
 \tag{Bilateral} \pc ab \longleftrightarrow a^u\setminus a^l \subseteq b^u \setminus b^l \\
 \tag{Lateral++} \pc ab \longleftrightarrow a^u\setminus a^l \subseteq b^l \\
 \tag{G-Simple} \pc ab \longleftrightarrow (\forall g\in \mathcal{G})(g\subseteq a \longrightarrow g\subseteq b)
 \end{align}

All these are valid concepts of parthoods that make sense in contexts as per availability and nature of information. \textsf{Very cautious} parthood makes sense in contexts with high cost of misclassification or  
the association between properties and objects is confounded by the lack of clarity in the possible set of properties. \textsf{G-Simple} is a version that refers granules alone and avoids references to approximations. 

The above mentioned list of parthoods can be more easily found in decision making contexts in practice. 

\begin{example}
Consider, for example, the nature of diagnosis and treatment of patients in a hospital in war torn Aleppo in the year 2016. The situation was characterised by shortage of medical personnel, damaged infrastructure, large number of patients and possibility of additional damage to infrastructure. Suppose patient \textsf{B} has bone fractures and a bullet embedded in their arm and patient \textsf{C} has bone fractures and shoulder dislocation due to a concrete slab in free fall, that only one doctor is on duty, and suppose that either of the two patients can be treated properly due to resource constraints. Suppose also that the doctor in question has access to some precise and unclear diagnostic information on medical conditions and that all of this data is not in tabular form.

In the situation, decision making can be based on available information and principles like
\begin{itemize}
\item {Allocate resources to the patient who is definitely in the worst state - this decision strategy can be corresponded to \emph{very cautious} parthoods,  }
\item {Allocate resources to the patient who seems to be in the worst state - this decision strategy can be corresponded to \emph{cautious} parthoods,}
\item {Allocate resources to the patient who is possibly in the worst state - this decision strategy can be corresponded to \emph{possibilist} parthoods.}
\item {Allocate resources to the patient who is likely to show more than the default amount of improvement - this decision strategy can be corresponded to the \emph{bilateral} parthoods.}
\item {If every symptom or unit complication that is experienced or certainly likely to be experienced by patient $1$ is also experienced or is certainly likely to be experienced by patient $2$, then prefer treating patient $2$ over patient $1$ - this decision strategy can be corresponded to the \emph{g-simple} parthoods with symptoms/unit complications as granules.}
\end{itemize} 
\end{example}

Parthood can be associated with both dialectic and dialethic statements in a number of ways. Cautious parthood is consistent with instances having the form $\pc ab$ and $\pc ba$. It is by itself a dialectic relation within the same domain of discourse. Dialectics between parthoods in different semantic domains are of greater interest and will be considered in subsequent sections.

The \emph{apparent parthood} relations considered in later sections of this paper typically arise from lack of clarity in specification of properties or due to imprecision (of fuzziness). This is illustrated in the next example.

\begin{example}[Cold vs Influenza]

Detection of influenza within $48$ hours of \emph{catching it} is necessary for effective treatment with anti-virals, but often patients fail to understand subtleties in distinguishing between cold and influenza. It is also not possible to administer comprehensive medical tests in a timely cost-effective way even in the best of facilities. So ground breaking insights even in restricted contexts can be useful.

The two medical conditions have similar symptoms. These may include \emph{fever} - as indicated by elevated temperatures, \emph{feverishness} - as indicated by personal experience (this may not be accompanied by fever), \emph{sneezing, running nose, blocked nose, headache of varying intensity,cough and body pain}. Body pain is usually a lot more intense in case of flu (but develops after a couple of days).

Clearly, in the absence of confirmatory tests patients can believe both \emph{instances of cold is apparently part of flu} and \emph{instances of flu is apparently part of cold}. These statements are in dialectical contradiction to each other but no dialetheias are involved. Cold and flu are also in dialectical contradiction to each other. This is a useful and relevant formalism.

It is another matter that if glutty negations are permitted then \emph{apparently part of} can as well be replaced by \emph{part of}.  
\end{example}

\section{Figures of Dialectical Opposition}

The scope of counting strategies and nature of possible models can be substantially improved when additional dialectical information about the nature of order-theoretic relation between rough and crisp objects is used. In the literature on generalizations of the square of opposition to rough sets, as in \cite{cd9,dc15}, it is generally assumed that realizations of such relations is the end result of semantic computations. This need not necessarily be so for reasons that will be explained below. In classical rough sets, a subset $X$ of objects $O$ results in a tri-partition of $O$ into the regions $L(X)$ (corresponding to lower approximation of $X$), $B(X)$ (the boundary region) and $E(X)$ (complement of the upper approximation). These form a hexagon of opposition (see Fig.\ref{paw}). In more general rough sets, this diagram generalises to cubes of opposition \cite{cd9}.

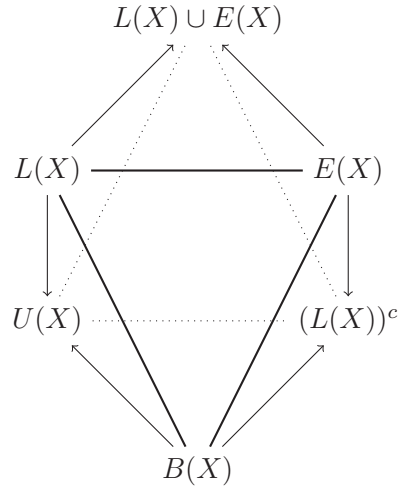
\begin{figure}[hbt]
\begin{center}
\begin{tikzpicture}[node distance=2cm, auto]
\node (A) {$L(X)\cup E(X)$};
\node (B) [below of=A] {};
\node (C) [left of=B] {$L(X)$};
\node (E) [right of=B] {$E(X)$};
\node (F) [below of=C] {$U(X)$};
\node (G) [below of=E] {$(L(X))^c$};
\node (O) [below of=B] {};
\node (H) [below of=O] {$B(X)$};
\draw[->,font=\scriptsize] (C) to node {}(A);
\draw[->,font=\scriptsize] (E) to node {}(A);
\draw[->,font=\scriptsize] (C) to node {}(F);
\draw[->,font=\scriptsize] (E) to node {}(G);
\draw[->,font=\scriptsize] (H) to node {}(F);
\draw[->,font=\scriptsize] (H) to node {}(G);
\draw[-,font=\scriptsize,dotted] (F) to node {}(A);
\draw[-,font=\scriptsize,dotted] (F) to node {}(G);
\draw[-,font=\scriptsize,dotted] (G) to node {}(A);
\draw[-,font=\scriptsize,thick] (C) to node {}(E);
\draw[-,font=\scriptsize,thick] (C) to node {}(H);
\draw[-,font=\scriptsize,thick] (E) to node {}(H);
\end{tikzpicture}
\caption{Hexagon of Opposition}
\label{paw}
\end{center}
\end{figure}

The general strategy used in a forthcoming paper is illustrated in \textsf{Fig.}\ref{stra}.

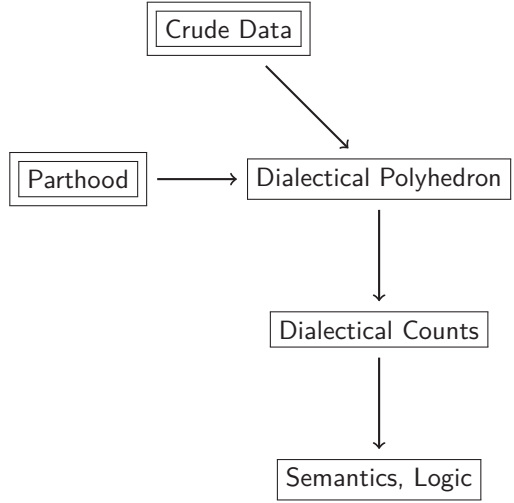
\begin{figure}[hbt]
\begin{center}
\begin{tikzpicture}[node distance=2cm, auto]
\node (A) {\fbox{\small{\fbox{\small{\textsf{Crude Data}}}}}};
\node (B) [below of=A] {};
\node (C) [left of=B]{\fbox{\small{\fbox{\small{\textsf{Parthood}}}}}};
\node (E) [right of=B] {\fbox{\small{\textsf{Dialectical Polyhedron}}}};
\node (F) [below of=E] {\fbox{\small{\textsf{Dialectical Counts}}}};
\node (K) [below of=F] {\fbox{\small{\textsf{Semantics, Logic}}}};
\draw[->,font=\scriptsize,thick] (A) to node {}(E);
\draw[->,font=\scriptsize,thick] (C) to node {}(E);
\draw[->,font=\scriptsize,thick] (E) to node {}(F);
\draw[->,font=\scriptsize,thick] (F) to node {}(K);
\end{tikzpicture}
\caption{Dialectical Semantics by Counting}
\label{stra}
\end{center}
\end{figure}

Some idea of parthood related ordering in the form of the following relations (a deeper understanding of ontology is essential for making sense of the vague usage (this is explored in \cite{am3690})) can suffice in application contexts of any dialectical generalised scheme of the square or hexagon of opposition (examples have already been provided earlier):

\begin{description}
\item[AP]{\textsf{Is Apparently Part of}: understood from class, property, expected behavior, or some other perspective.}
\item[APN]{\textsf{Is Apparently not Part of}. }
\item[AP0]{\textsf{Is Apparently Neither Part of Nor Not Part of}.}
\item[CP]{\textsf{Is Certainly Part of}.}
\item[CPN]{\textsf{Is Certainly Not Part of}.}
\item[CP0]{\textsf{Is Certainly Neither Part of Nor Not Part of}. (This is intended to convey uncertainty)}
\item[AI]{\textsf{Is Apparently Indistinguishable from}. }
\item[CI]{\textsf{Is Certainly Indistinguishable from}.}
\item[AW]{\textsf{Is Apparently a Whole of}}
\item[AWN]{\textsf{Is Apparently Not a Whole of}. }
\item[AW0]{\textsf{Is Apparently Neither a Whole of Nor Not a Whole of}.}
\item[CW]{\textsf{Is Certainly a Whole of}.}
\item[CWN]{\textsf{Is Certainly not a Whole of}.}
\item[CW0]{\textsf{Is Certainly Neither a Whole of Nor Not a Whole of}. }
\end{description}

By the word \textsf{apparently}, the agent may be referring to the lack of models, properties possessed by the objects, relativised views of the same among other possibilities. For example, the word \emph{apparently} can refer to the absence of any clear models about connections between diseases in data from a hospital chain in a single city or it can refer to problems caused by lack of data or relativizations about expected state of affairs relative to pre-existing models. Fuzzy and degree valuations of these perceptions are even less justified due to the use of approximate judgments. Predicates like \textsf{AP0, CP0} are needed for handling indecision (which is likely to be happen often in practice). 

As pointed out by a reviewer, \textsf{AP, APN} and \textsf{AP0} form three-fourths of Belnap's \emph{useful 4-valued logic} \cite{bn1977}. In this regard it should be noted that only those dialectical oppositions that can be sustained by inference procedures for progression of knowledge are relevant for logic.

The above set of predicates can be split into the following subsets of interest:
\begin{itemize}
\item{\textsf{Pure Apparence}: AP, APN, AP0, AW, AWN, AW0} 
\item{\textsf{Pure Certainty}: CP, CPN, CP0, CW, CWN, CW0}
\item{\textsf{Mixed Apparence}: AP, APN, AP0, AW, AWN, AW0, AI} 
\item{\textsf{Mixed Certainty}: CP, CPN, CP0, CW, CWN, CW0, CI}
\item{\textsf{Pure}: AP, APN, AP0, AW, AWN, AW0, CP, CPN, CP0, CW, CWN, CW0}
\item{\textsf{Mixed}: Union of all the above}
\end{itemize}

Before proceeding further it is necessary to fix the philosophical concepts of \emph{Contradiction, Contrariety, Sub-contrariety and Sub-alternation } because the literature on these concepts is vast and there is much scope for varying their meaning (see for example \cite{fs2012,ip}).  One way of looking at the connection between truth value assignments and sentences, necessary for the diagram to qualify in the square of opposition (generalised) paradigm, is illustrated in \textsf{Table.1, 2, 3, 4} below (NP means the assignment is not possible) :

\begin{table}[hbt]
\parbox{.45\linewidth}{
\centering
\begin{tabular}{|c|c|c|}
\hline
$\;\alpha \;$ & $\;\beta\;$ & $\mathbf{CY}(\alpha, \beta)\,$\\
\hline
T & T & NP\\
\hline
T & F & T \\
\hline
F & T & T \\ 
\hline
F & F & T \\
\hline
\end{tabular}
\caption{Contrariety}
}
\hfill
\parbox{.45\linewidth}{
\centering
\begin{tabular}{|c|c|c|}
\hline
$\; \alpha \;$ & \; $\beta\;$ & $\mathbf{\complement}(\alpha, \beta)\,$\\
\hline
T & T & NP \\
\hline
T & F & T \\
\hline
F & T & T \\
\hline
F & F & NP \\
\hline
\end{tabular}
\caption{Contradiction}
}
\end{table}

\begin{table}[hbt]
\parbox{.45\linewidth}{
\centering
\begin{tabular}{|c|c|c|}
\hline
$\;\alpha \;$ & $\;\beta\;$ & $\mathbf{SCY}(\alpha, \beta)\,$\\
\hline
T & T & T \\
\hline
T & F & T \\
\hline
F & T & T \\ 
\hline
F & F & NP \\
\hline
\end{tabular}
\caption{Sub-Contrariety}
}
\hfill
\parbox{.45\linewidth}{
\centering
\begin{tabular}{|c|c|c|}
\hline
$\; \alpha \;$ & \; $\beta\;$ & $\mathbf{AN}(\alpha, \beta)\,$\\
\hline
T & T & T \\
\hline
T & F & NP \\
\hline
F & T & T \\
\hline
F & F & T \\
\hline
\end{tabular}
\caption{SubAlternation}
}
\end{table}

The $P^Q$ semantics \cite{mo2} tries to take a simplified view of the situation. It may appear that the main problem with the proposal is a lack of suitable logical operators. But this drawback is not likely to be that significant for the counting based approach introduced in this paper and developed further in a forthcoming paper. The $P^Q$ approach in question is to look for answers to the questions:
\begin{itemize}
\item {\textsf{TT}: Can the sentences be true together?}
\item {\textsf{FT}: Can the sentences be false together?}
\end{itemize}

After finding those answers, categories can be worked out according to \textsf{Table.5}.

\begin{table}[hbt]
\centering
\begin{tabular}{|c|c|c|}
\hline
\diagbox[width=5em]{\textsf{TT}}{\textsf{FT}}&\thead{$\;\mathbf{1}\;$}&\thead{$\;\mathbf{0}\;$}\\
\hline
$\;\mathbf{1}\;$ & \textsf{Sub-alternation} & \textsf{Sub-Contrariety} \\
\hline
$\;\mathbf{0}\;$ & \textsf{Contrariety} & \textsf{Contradiction} \\
\hline
\end{tabular}
\caption{An Opposition}
\end{table}

But dialectical contradiction requires additional categories that relate to the following questions:
\begin{itemize}
\item {\textsf{Dialethia}: Can any one of the two statements be both true and false together? (Let $\delta(A)$ be the statement that $A$ is both false and true together).}
\item {\textsf{Bi-Dialectic}: Is either statement in dialectical opposition to the other? (Let $\beth (A, B)$ be the statement that $A$ is dialectically opposed to $B$).}
\item {\textsf{Dialectic}: Is either statement a statement expressing dialectical opposition? (Let $\beta (A)$ be the statement that $A$ expresses dialectical opposition with $\beta$ being a particular associated predicate). }
\end{itemize}

\emph{The above realization of the concept of dialetheia is pretty clear for implementation, but the latter two forms of dialectic depend on the choice of predicates and so many interpretations would be possible. In a typical concrete case, the parthood(s), the dialectical predicate and figure of opposition should be defined in order to obtain concrete answers. }

The Question-Answer Semantic approach (QAS) of \cite{fs2012} constitutes a relatively more complete strategy in which the \emph{sense} of a sentence $\alpha$ is an ordered set of questions $Q(\alpha) = \left\langle q_1(\alpha), \ldots , q_r(\alpha) \right\rangle$ and its \emph{reference} is an ordered set of answers 
$A(\alpha) = \left\langle a_1(\alpha), \ldots , a_r(\alpha) \right\rangle$. These answers can be coerced to binary form (with possible responses being \textsf{Yes} or \textsf{No}). 

If the dialectical approach of the present paper is extended to QAS approach, then the number of possible questions (like can question A and B be true together?) becomes very large and suitable subsets that are as efficient as the whole would be of interest. The possibilities are indicated in \textsf{Table.6}.

\begin{table}[hbt]
\centering
\begin{tabular}{l|rlrlrlrlrlrl}
\toprule

\textsf{ A} & T & F &$\beth$ &$\delta$ & $\delta$ & $\beta$ & $\beta$ &$\beta$& $\beta$& $\delta $ & $\delta $ & $\delta $ \\
\midrule
\textsf{ B} &T & F &$\beth$ & $\delta$ & $\beth$ &$\beta$ & $\beth$ & T & F& T & F & $\beta$\\

\bottomrule
\end{tabular}
\caption{Potential Combinations}
\end{table}

Connectives and operations can be involved in the definition of these predicates, but in the general case the meta concept of \emph{suitable subsets} can only be roughly estimated and not defined unambiguously. Even apparent-parthood related statements can be handled by the answer set after fixing the necessary subsets of the question set. In a separate paper, it is shown by the present author that counting procedures can be initiated with the help of this strategy.

Relative to the nature of truth values, two approaches to the problem can be adopted:
\begin{enumerate}
 \item {Keep the concept of truth and falsity fixed and attempt suitable definitions of oppositions or}
 \item {Permit variation of truth and falsity values. This in general would amount to deviating further from classical opposition paradigms.}
\end{enumerate}

\subsection{Classical Case-1: Fixed Truth }

When the concept of truth is not allowed to vary beyond the set $\{T, F\}$, then it is apparently possible to handle the cases involving apparent parthood without special external rules. A natural question that arises in the context of certain parthoods is about the admissibility of truth values. These aspects are considered in this subsection. 

One instructive (but not exhaustive) way is to read $\text{CP} ab $ as \textsf{all $a$ with property $\pi (a)$ and none of properties in $\neg\pi(a)$ (in the domain of discourse) are part of any $b$ with property $\pi(b)$}. As a consequence $\text{CPN} ab $ is \textsf{all $a$ with property $\pi (a)$ and none of properties in $\neg\pi(a)$ (in the domain of discourse) are not part of any $b$ with property $\pi(b)$}.

In rough sets, the association of objects with properties happen only when mechanisms of associations are explicitly specified or are specifiable. 
There is much freedom to choose from among different mechanisms of associations in a abstract perspective. 
In praxis, these choices become limited but rarely do they ever become absent. 
Implicit in all this is the assumption of stable choice among possible mechanisms of associations.
The stability aspect is an important research direction.

If truth tables for determining the nature of opposition is attempted using ideas of state resolution (instead of connectives) then perplexing results may happen. The choice of connectives is in turn hindered by an excess of choice. So the minimalist perspective based on two questions or the \textsf{QAS}-type approach should be preferred.

The following two theorems require interpretation.
\begin{theorem}
 The truth tables corresponding to two of the pairs 
 
 formed from \textsf{AP, APN, AP0} have the form indicated in \textsf{Table.7, 8}($P*Q$ is the resolution of the state relating to  $P$ and $Q$. This is abbreviated in the tables by $*$  ):
 
\begin{table}[hbt]
\parbox{.45\linewidth}{
\centering
\begin{tabular}{|c|c|c|}
\hline
$\;\text{AP} pq \;$ & $\;\text{APN}pq\;$ & $ * $\\
\hline
T & T & IN\\
\hline
T & F & T \\
\hline
F & T & T \\ 
\hline
F & F & IN \\
\hline
\end{tabular}
\caption{Contradiction?}
}
\hfill
\parbox{.45\linewidth}{
\centering
\begin{tabular}{|c|c|c|}
\hline
$\;\text{AP} pq \;$ & $\;\text{AP0}pq\;$ & $ * $\\
\hline
T & T & IN\\
\hline
T & F & T \\
\hline
F & T & T \\
\hline
F & F & IN \\
\hline
\end{tabular}
\caption{Contradiction?}
}
\end{table}
 
(\textsf{IN} is an abbreviation for indeterminate.) 
\end{theorem}

\begin{proof}
 The proof is direct. However, the interpretation is open and it is possible to read both tables as corresponding to contradiction.
 
\end{proof}

\begin{theorem}
 The truth tables corresponding to the pairs formed from \textsf{CP, CPN, CP0} and the pair \textsf{CP, CI} have the form indicated in 
 
 \textsf{Table.9,10,11,12} (\textsf{NP} is an abbreviation for not possible):
 
\begin{table}[hbt]
\parbox{.45\linewidth}{
\centering
\begin{tabular}{|c|c|c|}
\hline
$\;\text{CP} pq \;$ & $\;\text{CPN}pq\;$ & $* $\\
\hline
T & T & NP\\
\hline
T & F & T \\
\hline
F & T & T \\ 
\hline
F & F & NP \\
\hline
\end{tabular}
\caption{Contradiction}
}
\hfill
\parbox{.45\linewidth}{
\centering
\begin{tabular}{|c|c|c|}
\hline
$\;\text{CP} pq \;$ & $\;\text{CP0}pq\;$ & $*$\\
\hline
T & T & NP\\
\hline
T & F & T\\
\hline
F & T & T \\
\hline
F & F & T \\
\hline
\end{tabular}
\caption{Contrariety}}
\end{table}

\begin{table}[hbt]
\parbox{.45\linewidth}{
\centering
\begin{tabular}{|c|c|c|}
\hline
$\;\text{CPN} pq \;$ & $\;\text{CP0}pq\;$ & *\\
\hline
T & T & NP\\
\hline
T & F & T \\
\hline
F & T & T \\ 
\hline
F & F & NP \\
\hline
\end{tabular}
\caption{Contradiction}
}
\hfill
\parbox{.45\linewidth}{
\centering
\begin{tabular}{|c|c|c|}
\hline
$\;\text{CI} pq \;$ & $\;\text{CP}pq\;$ & $*$\\
\hline
T & T & T\\
\hline
T & F & NP\\
\hline
F & T & T \\
\hline
F & F & T \\
\hline
\end{tabular}
\caption{Sub-alternation}
}
\end{table}
\end{theorem}

\begin{proof}
The proof consists in checking the possibilities by cases. 
In the table for \textsf{CP} and \textsf{CP0}, the last line is justified because no possibilities are covered by the last column. 
\end{proof}

From the safer (and questionable) two question framework, the above two theorems have the following form:

\begin{theorem}
The answers to the two simultaneity (\textsf{Sim}) questions for the pairs formed from \textsf{AP, APN, AP0} are in \textsf{Table.13,14}.
 
\begin{table}[hbt]
\parbox{.45\linewidth}{
\centering
\begin{tabular}{|c|c|c|}
\hline
$\;\text{AP} ab \;$ & $\;\text{APN}ab\;$ & \textsf{Sim}\\
\hline
T & T & NP\\
\hline
F & F & NP \\
\hline
\end{tabular}
\caption{Contradiction}
}
\hfill
\parbox{.45\linewidth}{
\centering
\begin{tabular}{|c|c|c|}
\hline
$\;\text{AP} ab \;$ & $\;\text{AP0}ab\;$ & \textsf{Sim} \\
\hline
T & T & T\\
\hline
F & F & NP \\
\hline
\end{tabular}
\caption{Sub-Contrariety}
}
\end{table}
 
(\textsf{NP} is an abbreviation for not possible.) 
\end{theorem}
\begin{proof}
In the table for \textsf{AP, AP0}, TT discludes all possibilities  and therefore yields T. Other parts are not hard to prove.

\end{proof}

\begin{theorem}
The simultaneity data corresponding to the pairs

formed from \textsf{CP, CPN, CP0} are in \textsf{Table.15,16,17,18}.
 
\begin{table}[hbt]
\parbox{.45\linewidth}{
\centering
\begin{tabular}{|c|c|c|}
\hline
$\;\text{CP} ab \;$ & $\;\text{CPN}ab\;$ & \textsf{Sim} \\
\hline
T & T & NP\\
\hline
F & F & NP \\
\hline
\end{tabular}
\caption{Contradiction}
}
\hfill
\parbox{.45\linewidth}{
\centering
\begin{tabular}{|c|c|c|}
\hline
$\;\text{CP} ab \;$ & $\;\text{CP0}ab\;$ & \textsf{Sim}\\
\hline
T & T & NP\\
\hline
F & F & NP \\
\hline
\end{tabular}
\caption{Contradiction}
}
\end{table}

\begin{table}[hbt]
\parbox{.45\linewidth}{
\centering
\begin{tabular}{|c|c|c|}
\hline
$\;\text{CPN} ab \;$ & $\;\text{CP0}ab\;$ & \textsf{Sim}\\
\hline
T & T & NP\\
\hline
F & F & NP \\
\hline
\end{tabular}
\caption{Contradiction}
}
\hfill
\parbox{.45\linewidth}{
\centering
\begin{tabular}{|c|c|c|}
\hline
$\;\text{CI} ab \;$ & $\;\text{CP}ab\;$ & \textsf{Sim}\\
\hline
T & T & T\\
\hline
F & F & T \\
\hline
\end{tabular}
\caption{Sub-alternation}
}
\end{table}

\end{theorem}
\begin{proof}
 In Table 16, for example, the question is can both $\text{CPab}$ be false and $\text{CP0ab}$ be false? As the situation is impossible, NP is the result. 
 
\end{proof}

\subsection{Case-2: Pseudo Gluts}

A minimalist use of assumptions on possible grades of truth in the cases admitting apparent parthood leads to the following diagram of truth values. The figure is biased against falsity because in the face of contradiction agents are expected to be truth seeking - this admittedly is a potentially contestable philosophical statement.

\begin{figure}[hbt]
\begin{center}
\begin{tikzpicture}[node distance=2cm, auto]
\node (A) {\fbox{${\mathbf T^*}$}};
\node (B) [below of=A] {\fbox{${\mathbf T_*}$}};
\node (C) [below of=B] {\fbox{${\mathbf T}$}};
\node (E) [right of=C] {};
\node (F) [left of=C] {};
\node (G) [below of=F] {\fbox{${\mathbf F^{\ominus}}$}};
\node (H) [below of=G] {\fbox{${\mathbf F_{\ominus}}$}};
\node (J) [below of=E] {\fbox{${\mathbf T^{\ominus}}$}};
\node (K) [below of=J] {\fbox{${\mathbf T_{\ominus}}$}};
\node (L) [left of=K] {};
\node (N) [below of=L] {\fbox{${\mathbf F}$}};
\draw[->,font=\scriptsize,thick] (A) to node {}(B);
\draw[->,font=\scriptsize,thick] (B) to node {}(C);
\draw[->,font=\scriptsize,thick] (C) to node {}(G);
\draw[->,font=\scriptsize,thick] (C) to node {}(J);
\draw[->,font=\scriptsize,thick] (G) to node {}(H);
\draw[->,font=\scriptsize,thick] (J) to node {}(K);
\draw[->,font=\scriptsize,thick] (K) to node {}(N);
\draw[->,font=\scriptsize,thick] (H) to node {}(N);
\end{tikzpicture}
\caption{Weak and Strong Truths}\label{glut}
\end{center}
\end{figure}
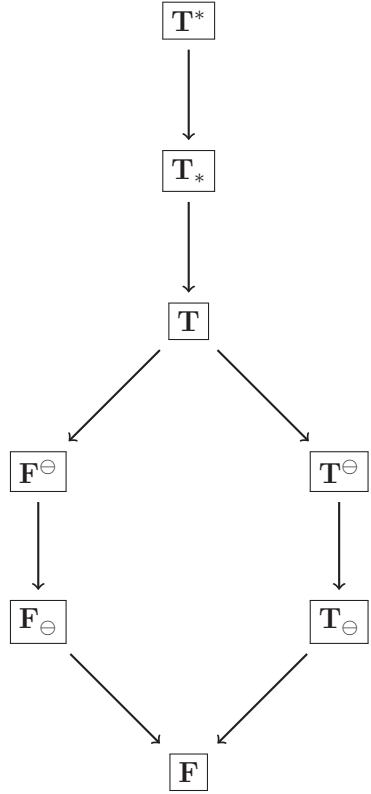

Reading of truth tables in relation to state transition based conjunction is also relevant for dialectical interpretation. But these are not handled by the above tables and will be part of future work.

For the dialectical counting procedures introduced in the next section, the basic contexts are assumed to be very minimalist and possibly naive. In these some meta principles on aggregation of truth can be useful or natural. The states of truth mentioned in \textsf{Fig.}\ref{glut} relate to the following meta method of handling apparent truth. These will be referred as \emph{Truth State Determining Rules} (\textsf{TSR}). In the rules $\alpha, \beta$ are intended in particular for formulas having the form $\pc ab$ and variants.

\begin{flushleft}
\textsf{Truth State Determining Rules} 
\end{flushleft}

\begin{itemize}
\item {\textsf{If $\alpha$ is apparently true and $\beta$ supports it, then $\alpha $ becomes more true.}}
\item {\textsf{If $\alpha$ is apparently true and $\beta$ opposes it, then $\alpha $ becomes less true.}}
\item {\textsf{If $\alpha$ is less true (than true is supposed to be) and $\beta$ opposes it, then $\alpha $ becomes even less true.}}
\item {\textsf{If $\alpha$ is apparently false and $\beta$ opposes it, then $\alpha $ becomes less false.}}
\item {\textsf{In the figure, $\text{T}$ denotes an intermediate truth value that can become stronger $\text{T}_*, \, \text{T}^*$ or weaker $\text{T}_\ominus , \text{T}^\ominus$.} This is because operators (apparently) like \emph{less} and \emph{even less} are available.}
\end{itemize}

The above list of rules can be made precise using the distance between vertices in the graph and thus it would be possible to obtain truth values associated with combinations of sentences involving apparent parthood alone.

\subsection{Counting Procedures and Dialectical Opposition}

A brief introduction to the dialectical counting procedures and semantics invented by the present author is included in this subsection. The full version will appear in a separate paper. 

New concepts of rough natural number systems were introduced in \cite{am240} from both formal and less-formal perspectives. These are useful for improving most rough set-theoretical measures in general rough sets and in the representation of rough semantics. In particular it was proved that the algebraic semantics of classical \textsf{RST} can be obtained from the developed dialectical counting procedures. In the counting contexts of \cite{am240}, a pair of integers under contextual rules suffices to indicate the \emph{number} associated with the element in an instance of the counting scheme under consideration. This is because those processes basically use a \emph{square of discernibility} with the statements at vertices having the following form:

\begin{itemize}
\item {$\text{ IS.NOT}(a,b)$ meaning $a$ is not $b$. }
\item {$\text{ IS}(a,b)$ meaning $a$ is identical with $b$.}
\item {$\text{ IND}(a,b)$ meaning $a$ is indiscernible from $b$.}
\item {$\text{ DIS}(a,b)$ meaning $a$ is discernible from $b$.}
\end{itemize}

\begin{figure}[hbt]
\begin{center}
\begin{tikzpicture}[node distance=4cm, auto]
\node (A) {\fbox{$\text{IS.NOT}(a,b)$}};
\node (B) [below of=A] {\fbox{$\text{IND}(a,b)$}};
\node (C) [right of=B] {\fbox{$\text{IS}(a,b)$}};
\node (E) [right of=A] {\fbox{$\text{DIS}(a,b)$}};
\draw[->,font=\scriptsize,thick] (E) to node {}(A);
\draw[->,font=\scriptsize,thick] (B) to node {}(C);
\draw[-,font=\scriptsize] (A) to node {}(B);
\draw[-,font=\scriptsize] (C) to node {}(E);
\draw[-,font=\scriptsize,dashed] (B) to node {$\beth$}(E);
\end{tikzpicture}
\caption{Rough Counting}\label{rc}
\end{center}
\end{figure}
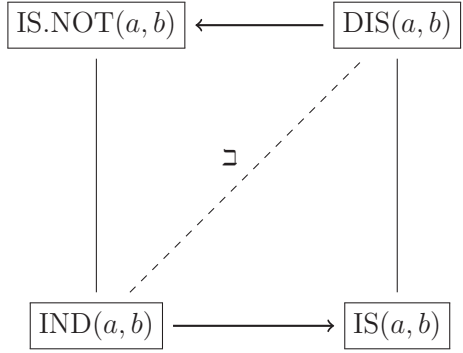

A particular example of such a counting method is the \emph{indiscernible predecessor based primitive counting} (IPC). The adjective \emph{primitive} is intended to express a minimal use of granularity and related axioms. Let $S = \{x_{1},\,x_{2},\,\ldots,\,x_{k},\,\ldots ,\,\}$ be an at most countable set of elements
in ZF that correspond to some collection of objects.  If $R$ is a binary relation on $S$ that corresponds to \emph{is weakly indiscernible from}, then IPC takes the following form:  
 
\paragraph{Indiscernible Predecessor Based Primitive Counting (IPC)}

In this form of 'counting', the relation with the immediate preceding step of counting 
matters in a crucial way.

\begin{enumerate}
\item {Assign $f(x_{1})\, = \,1_{1} =s^{0}(1_{1})$.}
\item {If $f(x_{i})=s^{r}(1_{j})$ and $(x_{i},x_{i+1}) \in  R$, then assign
$f(x_{i+1})=1_{j+1}$.}
\item {If $f(x_{i})=s^{r}(1_{j})$ and $(x_{i},x_{i+1}) \notin  R$, then assign
$f(x_{i+1})=s^{r+1}(1_{j})$.}
\end{enumerate}

For example, if $S=\{f, b, c, a, k, i, n, h, e, l, g, m \}$ and $R$ is the reflexive and transitive
closure of the relation \[\{(a,b),\,(b, c),\,(e, f),\, (i, k), (l, m),\, (m,
n),\,(g, h) \}\] then $S$ can be counted in the presented order as below.
\[\tag{IPC}{\{1_{1}, 2_{1}, 1_{2},1_{3}, 2_{3}, 1_{4}, 2_{4}, 3_{4}, 1_{5}, 2_{5},1_{6}, 2_{6} \},} \]

Rough counting methods like the above have nice algebraic models and can also be used for representing semantics.
The dialectical approaches of the present paper motivate generalizations based on the following:
\begin{itemize}
\item {Counting by Dialectical Mereology: This method of counting is intended to be based on the principle that the mereological relation of the object being counted with its predecessors  should determine its count and the enumeration should be on convex regular polygons (including squares and hexagons of opposition), polyhedrons or generalizations thereof (polytopes) of dialectical and classical opposition. }
\item {Counting by Threes: This is based on the principle that the relation between  the object being counted with its predecessor and the mereological relation of the object with its successor should determine its count. The approach admits of many variations.}
\item {Counting by Reduction to Discernibility: It is also possible to count taking increasing scopes of discernibility  into account.}
\end{itemize}

\section*{Further Directions and Remarks}

In this research paper all the following have been accomplished by the present author (apart from contributing to the  solution of some inverse problem contexts)

\begin{itemize}
\item {Formalization of possible concepts of dialectical contradiction has been done in one possible way using object level predicates.}
\item {The difference between dialetheias and dialectical contradiction has been clarified. It has been argued that dialectical contradiction need not be reducible to dialetheias or be associated with glutty negation. But the latter correspond to Hegelian dialectical contradiction.}
\item {A pair of dialectical rough semantics have been developed over classical rough sets. The nature of parthood is reexamined and used for counting based approaches to semantics.}
\item {Opposition in the context of rough set and parthood related sentences is investigated and concepts of dialectical opposition and opposition are generated. Related truth tables show that the classical figures do not work as well for parthood related sentences. This extends previous work on figures of opposition of rough sets in new directions. A major contribution is in the use of dialectical negation predicates in the generation of figures of opposition.}
\item {Possible diagrams of opposition are used for defining generalised counting process for constructive algebraic semantics. This builds on earlier work of the present author in \cite{am240}.}
\item {The antichain based semantics for general rough sets that has been developed in \cite{am6999,am9114} has been supplemented with a constructive dialectical counting process and scope for using generalised negations in a separate paper. The foundations for the same has been laid in this paper.}
\end{itemize}

Sub-classes of problems that have been motivated by the present paper relate to 

\begin{itemize}
\item {Formal characterization of conditions that would permit reduction of dialectical contradiction to dialetheia.}
\item {Construction of algebraic semantics from dialectical counting using figures of opposition.}
\item {Construction of general rough set algorithms from the antichain based dialectics.}
\item {Algebraic characterization of parthood based semantics as opposed to rough object based approaches. This problem has been substantially solved by the present author in \cite{am9114}.}
\item {Methods of property extraction by formal interpretation or translation across semantics - this aspect has not been discussed in detail in the present paper and will appear separately.}
\item {Development of logics relating to glutty negation in the context of suitable rough contexts.}
\end{itemize}

\begin{flushleft}
\textbf{Acknowledgement}:
\end{flushleft}
\begin{small}
The present author would like to thank the referees for detailed remarks that led to improvement (especially of the readability) of the research paper. 
\end{small}

\bibliographystyle{apacite}
\bibliography{algroughffe.bib}

\begin{thebibliography}{}

\bibitem [\protect \citeauthoryear {%
Apostol%
}{%
Apostol%
}{%
{\protect \APACyear {1979}}%
}]{%
al}
\APACinsertmetastar {%
al}%
\begin{APACrefauthors}%
Apostol, L.%
\end{APACrefauthors}%
\unskip\
\newblock
\APACrefYear{1979}.
\newblock
\APACrefbtitle {{Logique et Dialectique}} {{Logique et Dialectique}}.
\newblock
\APACaddressPublisher{}{Gent}.
\PrintBackRefs{\CurrentBib}

\bibitem [\protect \citeauthoryear {%
Avron%
\ \BBA {} Konikowska%
}{%
Avron%
\ \BBA {} Konikowska%
}{%
{\protect \APACyear {2008}}%
}]{%
ak}
\APACinsertmetastar {%
ak}%
\begin{APACrefauthors}%
Avron, A.%
\BCBT {}\ \BBA {} Konikowska, B.%
\end{APACrefauthors}%
\unskip\
\newblock
\APACrefYearMonthDay{2008}{}{}.
\newblock
{\BBOQ}\APACrefatitle {{Rough Sets and 3-Valued Logics}} {{Rough Sets and
  3-Valued Logics}}.{\BBCQ}
\newblock
\APACjournalVolNumPages{Studia Logica}{90}{}{69--92}.
\PrintBackRefs{\CurrentBib}

\bibitem [\protect \citeauthoryear {%
Banerjee%
\ \BBA {} Chakraborty%
}{%
Banerjee%
\ \BBA {} Chakraborty%
}{%
{\protect \APACyear {1996}}%
}]{%
bc1}
\APACinsertmetastar {%
bc1}%
\begin{APACrefauthors}%
Banerjee, M.%
\BCBT {}\ \BBA {} Chakraborty, M\BPBI K.%
\end{APACrefauthors}%
\unskip\
\newblock
\APACrefYearMonthDay{1996}{}{}.
\newblock
{\BBOQ}\APACrefatitle {{Rough Sets Through Algebraic Logic}} {{Rough Sets
  Through Algebraic Logic}}.{\BBCQ}
\newblock
\APACjournalVolNumPages{Fundamenta Informaticae}{28}{}{211--221}.
\PrintBackRefs{\CurrentBib}

\bibitem [\protect \citeauthoryear {%
Banerjee%
, Chakraborty%
\BCBL {}\ \BBA {} Bunder%
}{%
Banerjee%
\ \protect \BOthers {.}}{%
{\protect \APACyear {2008}}%
}]{%
bcb}
\APACinsertmetastar {%
bcb}%
\begin{APACrefauthors}%
Banerjee, M.%
, Chakraborty, M\BPBI K.%
\BCBL {}\ \BBA {} Bunder, M.%
\end{APACrefauthors}%
\unskip\
\newblock
\APACrefYearMonthDay{2008}{}{}.
\newblock
{\BBOQ}\APACrefatitle {{Some Rough Consequence Logics and Their
  Interrelations}} {{Some Rough Consequence Logics and Their
  Interrelations}}.{\BBCQ}
\newblock
\BIn{} A.~Skowron\ \BBA {} J\BPBI F.~Peters\ (\BEDS), \APACrefbtitle
  {{Transactions on Rough Sets VIII}} {{Transactions on Rough Sets VIII}}\
  (\BVOL\ LNCS 5084, \BPGS\ 1--20).
\newblock
\APACaddressPublisher{}{Springer Verlag}.
\PrintBackRefs{\CurrentBib}

\bibitem [\protect \citeauthoryear {%
Batens%
}{%
Batens%
}{%
{\protect \APACyear {1989}}%
}]{%
bd}
\APACinsertmetastar {%
bd}%
\begin{APACrefauthors}%
Batens, D.%
\end{APACrefauthors}%
\unskip\
\newblock
\APACrefYearMonthDay{1989}{}{}.
\newblock
{\BBOQ}\APACrefatitle {{Dynamic Dialectical Logics}} {{Dynamic Dialectical
  Logics}}.{\BBCQ}
\newblock
\BIn{} \APACrefbtitle {{Paraconsistent Logic - Essays}.} {{Paraconsistent Logic
  - Essays}.}
\newblock
\APACaddressPublisher{Munich}{Philosophia Verlag}.
\PrintBackRefs{\CurrentBib}

\bibitem [\protect \citeauthoryear {%
Batens%
}{%
Batens%
}{%
{\protect \APACyear {1990}}%
}]{%
bd2}
\APACinsertmetastar {%
bd2}%
\begin{APACrefauthors}%
Batens, D.%
\end{APACrefauthors}%
\unskip\
\newblock
\APACrefYearMonthDay{1990}{}{}.
\newblock
{\BBOQ}\APACrefatitle {{Against Global Paraconsistency}} {{Against Global
  Paraconsistency}}.{\BBCQ}
\newblock
\APACjournalVolNumPages{Studies in Soviet Thought}{39}{}{209--229}.
\PrintBackRefs{\CurrentBib}

\bibitem [\protect \citeauthoryear {%
Batens%
}{%
Batens%
}{%
{\protect \APACyear {2006}}%
}]{%
bd1}
\APACinsertmetastar {%
bd1}%
\begin{APACrefauthors}%
Batens, D.%
\end{APACrefauthors}%
\unskip\
\newblock
\APACrefYearMonthDay{2006}{}{}.
\newblock
{\BBOQ}\APACrefatitle {{Narrowing Down Suspicion in Inconsistent Premise Set}}
  {{Narrowing Down Suspicion in Inconsistent Premise Set}}{\BBCQ}\
  [\bibcomputersoftwaremanual].
\newblock
\APACaddressPublisher{Preprint}{}.
\PrintBackRefs{\CurrentBib}

\bibitem [\protect \citeauthoryear {%
Belnap%
}{%
Belnap%
}{%
{\protect \APACyear {1977}}%
}]{%
bn1977}
\APACinsertmetastar {%
bn1977}%
\begin{APACrefauthors}%
Belnap, N\BPBI D.%
\end{APACrefauthors}%
\unskip\
\newblock
\APACrefYearMonthDay{1977}{}{}.
\newblock
{\BBOQ}\APACrefatitle {{A Useful Four-Valued Logic}} {{A Useful Four-Valued
  Logic}}.{\BBCQ}
\newblock
\BIn{} J\BPBI M.~Dunn\ \BBA {} G.~Epstein\ (\BEDS), \APACrefbtitle {{Modern
  Uses of Multiple-Valued Logic}} {{Modern Uses of Multiple-Valued Logic}}\
  (\BPGS\ 5--37).
\newblock
\APACaddressPublisher{Dordrecht}{Springer Netherlands}.
\newblock
\begin{APACrefDOI} \doi{10.1007/978-94-010-1161-7_2} \end{APACrefDOI}
\PrintBackRefs{\CurrentBib}

\bibitem [\protect \citeauthoryear {%
Brandom%
}{%
Brandom%
}{%
{\protect \APACyear {2008}}%
}]{%
br2008}
\APACinsertmetastar {%
br2008}%
\begin{APACrefauthors}%
Brandom, R.%
\end{APACrefauthors}%
\unskip\
\newblock
\APACrefYear{2008}.
\newblock
\APACrefbtitle {{Between Saying and Doing}} {{Between Saying and Doing}}.
\newblock
\APACaddressPublisher{}{Oxford University Press}.
\PrintBackRefs{\CurrentBib}

\bibitem [\protect \citeauthoryear {%
Cattaneo%
}{%
Cattaneo%
}{%
{\protect \APACyear {2018}}%
}]{%
gc2018}
\APACinsertmetastar {%
gc2018}%
\begin{APACrefauthors}%
Cattaneo, G.%
\end{APACrefauthors}%
\unskip\
\newblock
\APACrefYearMonthDay{2018}{}{}.
\newblock
{\BBOQ}\APACrefatitle {{Algebraic Methods for Rough Approximation Spaces by
  Lattice Interior--closure Operations}} {{Algebraic Methods for Rough
  Approximation Spaces by Lattice Interior--closure Operations}}.{\BBCQ}
\newblock
\BIn{} A.~Mani, I.~D{\"u}ntsch\BCBL {}\ \BBA {} G.~Cattaneo\ (\BEDS),
  \APACrefbtitle {{Algebraic Methods in General Rough Sets}} {{Algebraic
  Methods in General Rough Sets}}\ (\BPGS\ 12--124).
\newblock
\APACaddressPublisher{}{Springer International}.
\PrintBackRefs{\CurrentBib}

\bibitem [\protect \citeauthoryear {%
Cattaneo%
\ \BBA {} Ciucci%
}{%
Cattaneo%
\ \BBA {} Ciucci%
}{%
{\protect \APACyear {2004}}%
}]{%
cc}
\APACinsertmetastar {%
cc}%
\begin{APACrefauthors}%
Cattaneo, G.%
\BCBT {}\ \BBA {} Ciucci, D.%
\end{APACrefauthors}%
\unskip\
\newblock
\APACrefYearMonthDay{2004}{}{}.
\newblock
{\BBOQ}\APACrefatitle {{Algebras for Rough Sets and Fuzzy Logics}} {{Algebras
  for Rough Sets and Fuzzy Logics}}.{\BBCQ}
\newblock
\APACjournalVolNumPages{Transactions on Rough Sets}{2}{LNCS 3100}{208--252}.
\PrintBackRefs{\CurrentBib}

\bibitem [\protect \citeauthoryear {%
Cattaneo%
\ \BBA {} Ciucci%
}{%
Cattaneo%
\ \BBA {} Ciucci%
}{%
{\protect \APACyear {2009}}%
}]{%
cc5}
\APACinsertmetastar {%
cc5}%
\begin{APACrefauthors}%
Cattaneo, G.%
\BCBT {}\ \BBA {} Ciucci, D.%
\end{APACrefauthors}%
\unskip\
\newblock
\APACrefYearMonthDay{2009}{}{}.
\newblock
{\BBOQ}\APACrefatitle {{Lattices With Interior and Closure Operators and
  Abstract Approximation Spaces}} {{Lattices With Interior and Closure
  Operators and Abstract Approximation Spaces}}.{\BBCQ}
\newblock
\BIn{} J\BPBI F.~Peters\ \BOthers {.}\ (\BEDS), \APACrefbtitle {{Transactions
  on Rough Sets X, LNCS 5656}} {{Transactions on Rough Sets X, LNCS 5656}}\
  (\BPGS\ 67--116).
\newblock
\APACaddressPublisher{}{Springer}.
\PrintBackRefs{\CurrentBib}

\bibitem [\protect \citeauthoryear {%
Cattaneo%
\ \BBA {} Ciucci%
}{%
Cattaneo%
\ \BBA {} Ciucci%
}{%
{\protect \APACyear {2018}}%
}]{%
gcd2018}
\APACinsertmetastar {%
gcd2018}%
\begin{APACrefauthors}%
Cattaneo, G.%
\BCBT {}\ \BBA {} Ciucci, D.%
\end{APACrefauthors}%
\unskip\
\newblock
\APACrefYearMonthDay{2018}{}{}.
\newblock
{\BBOQ}\APACrefatitle {{Algebraic Methods for Orthopairs and induced Rough
  Approximation Spaces}} {{Algebraic Methods for Orthopairs and induced Rough
  Approximation Spaces}}.{\BBCQ}
\newblock
\BIn{} A.~Mani, I.~D{\"u}ntsch\BCBL {}\ \BBA {} G.~Cattaneo\ (\BEDS),
  \APACrefbtitle {{Algebraic Methods in General Rough Sets}} {{Algebraic
  Methods in General Rough Sets}}\ (\BPGS\ 534--623).
\PrintBackRefs{\CurrentBib}

\bibitem [\protect \citeauthoryear {%
Cattaneo%
, Ciucci%
\BCBL {}\ \BBA {} Dubois%
}{%
Cattaneo%
\ \protect \BOthers {.}}{%
{\protect \APACyear {2011}}%
}]{%
ccd11}
\APACinsertmetastar {%
ccd11}%
\begin{APACrefauthors}%
Cattaneo, G.%
, Ciucci, D.%
\BCBL {}\ \BBA {} Dubois, D.%
\end{APACrefauthors}%
\unskip\
\newblock
\APACrefYearMonthDay{2011}{}{}.
\newblock
{\BBOQ}\APACrefatitle {{Algebraic Models of Deviant Modal Operators Based on De
  Morgan and Kleene Lattices}} {{Algebraic Models of Deviant Modal Operators
  Based on De Morgan and Kleene Lattices}}.{\BBCQ}
\newblock
\APACjournalVolNumPages{Information Sciences}{181}{}{4075--4100}.
\newblock
\begin{APACrefDOI} \doi{10.1016/J.INS.2011.05.008} \end{APACrefDOI}
\PrintBackRefs{\CurrentBib}

\bibitem [\protect \citeauthoryear {%
Chakraborty%
}{%
Chakraborty%
}{%
{\protect \APACyear {2016}}%
}]{%
mkcfi2016}
\APACinsertmetastar {%
mkcfi2016}%
\begin{APACrefauthors}%
Chakraborty, M\BPBI K.%
\end{APACrefauthors}%
\unskip\
\newblock
\APACrefYearMonthDay{2016}{}{}.
\newblock
{\BBOQ}\APACrefatitle {{On Some Issues in the Foundation of Rough Sets: the
  Problem of Definition}} {{On Some Issues in the Foundation of Rough Sets: the
  Problem of Definition}}.{\BBCQ}
\newblock
\APACjournalVolNumPages{Fundamenta Informaticae}{148}{}{123--132}.
\PrintBackRefs{\CurrentBib}

\bibitem [\protect \citeauthoryear {%
Ciucci%
}{%
Ciucci%
}{%
{\protect \APACyear {2009}}%
}]{%
cd3}
\APACinsertmetastar {%
cd3}%
\begin{APACrefauthors}%
Ciucci, D.%
\end{APACrefauthors}%
\unskip\
\newblock
\APACrefYearMonthDay{2009}{}{}.
\newblock
{\BBOQ}\APACrefatitle {{Approximation Algebra and Framework}} {{Approximation
  Algebra and Framework}}.{\BBCQ}
\newblock
\APACjournalVolNumPages{Fundamenta Informaticae}{94}{}{147--161}.
\PrintBackRefs{\CurrentBib}

\bibitem [\protect \citeauthoryear {%
Ciucci%
}{%
Ciucci%
}{%
{\protect \APACyear {2016}}%
}]{%
dc15}
\APACinsertmetastar {%
dc15}%
\begin{APACrefauthors}%
Ciucci, D.%
\end{APACrefauthors}%
\unskip\
\newblock
\APACrefYearMonthDay{2016}{}{}.
\newblock
{\BBOQ}\APACrefatitle {{Orthopairs and Granular Computing }} {{Orthopairs and
  Granular Computing }}.{\BBCQ}
\newblock
\APACjournalVolNumPages{Granular Computing}{In Press}{}{1--12}.
\newblock
\APAChowpublished {Online}.
\newblock
\begin{APACrefDOI} \doi{10.1007/s41066-015-0013-y} \end{APACrefDOI}
\PrintBackRefs{\CurrentBib}

\bibitem [\protect \citeauthoryear {%
Ciucci%
}{%
Ciucci%
}{%
{\protect \APACyear {2017}}%
}]{%
cd2017}
\APACinsertmetastar {%
cd2017}%
\begin{APACrefauthors}%
Ciucci, D.%
\end{APACrefauthors}%
\unskip\
\newblock
\APACrefYearMonthDay{2017}{}{}.
\newblock
{\BBOQ}\APACrefatitle {{Back To The Beginnings: Pawlak'S Definitions of The
  Terms Information System and Rough Set}} {{Back To The Beginnings: Pawlak'S
  Definitions of The Terms Information System and Rough Set}}.{\BBCQ}
\newblock
\BIn{} G.~Wang\ \BOthers {.}\ (\BEDS), \APACrefbtitle {{Thriving Rough Sets}}
  {{Thriving Rough Sets}}\ (\BPGS\ 225--236).
\newblock
\APACaddressPublisher{}{Springer International}.
\PrintBackRefs{\CurrentBib}

\bibitem [\protect \citeauthoryear {%
Ciucci%
, Dubois%
\BCBL {}\ \BBA {} Prade%
}{%
Ciucci%
\ \protect \BOthers {.}}{%
{\protect \APACyear {2012}}%
}]{%
cd9}
\APACinsertmetastar {%
cd9}%
\begin{APACrefauthors}%
Ciucci, D.%
, Dubois, D.%
\BCBL {}\ \BBA {} Prade, H.%
\end{APACrefauthors}%
\unskip\
\newblock
\APACrefYearMonthDay{2012}{}{}.
\newblock
{\BBOQ}\APACrefatitle {{Oppositions in Rough Set Theory}} {{Oppositions in
  Rough Set Theory}}.{\BBCQ}
\newblock
\BIn{} T.~Li\ \BOthers {.}\ (\BEDS), \APACrefbtitle {{RSKT'2012, LNAI 7414}}
  {{RSKT'2012, LNAI 7414}}\ (\BPGS\ 504--513).
\newblock
\APACaddressPublisher{}{Springer-Verlag}.
\PrintBackRefs{\CurrentBib}

\bibitem [\protect \citeauthoryear {%
{da Costa}%
\ \BBA {} Wolf%
}{%
{da Costa}%
\ \BBA {} Wolf%
}{%
{\protect \APACyear {1974}}%
}]{%
dw}
\APACinsertmetastar {%
dw}%
\begin{APACrefauthors}%
{da Costa}, N.%
\BCBT {}\ \BBA {} Wolf, R\BPBI G.%
\end{APACrefauthors}%
\unskip\
\newblock
\APACrefYearMonthDay{1974}{}{}.
\newblock
{\BBOQ}\APACrefatitle {{Studies in Paraconsistent Logic-1 : The Dialectical
  Principle of The Unity of Opposites}} {{Studies in Paraconsistent Logic-1 :
  The Dialectical Principle of The Unity of Opposites}}.{\BBCQ}
\newblock
\APACjournalVolNumPages{Philosophia - Philosophical Quarterly of
  Israel}{15}{}{497--510}.
\PrintBackRefs{\CurrentBib}

\bibitem [\protect \citeauthoryear {%
D{\"u}ntsch%
\ \BBA {} Gediga%
}{%
D{\"u}ntsch%
\ \BBA {} Gediga%
}{%
{\protect \APACyear {2000}}%
}]{%
gdu}
\APACinsertmetastar {%
gdu}%
\begin{APACrefauthors}%
D{\"u}ntsch, I.%
\BCBT {}\ \BBA {} Gediga, G.%
\end{APACrefauthors}%
\unskip\
\newblock
\APACrefYear{2000}.
\newblock
\APACrefbtitle {{Rough set data analysis: A road to non-invasive knowledge
  discovery}} {{Rough set data analysis: A road to non-invasive knowledge
  discovery}}.
\newblock
\APACaddressPublisher{}{Methodos Publishers}.
\PrintBackRefs{\CurrentBib}

\bibitem [\protect \citeauthoryear {%
Ficara%
}{%
Ficara%
}{%
{\protect \APACyear {2014}}%
}]{%
ef14}
\APACinsertmetastar {%
ef14}%
\begin{APACrefauthors}%
Ficara, E.%
\end{APACrefauthors}%
\unskip\
\newblock
\APACrefYearMonthDay{2014}{}{}.
\newblock
{\BBOQ}\APACrefatitle {{Hegel{\rq}s Glutty Negation}} {{Hegel{\rq}s Glutty
  Negation}}.{\BBCQ}
\newblock
\APACjournalVolNumPages{History and Philosophy of Logic}{}{}{1--10}.
\newblock
\begin{APACrefURL} \url{http://dx.doi.org/10.1080/01445340.2014.940698}
  \end{APACrefURL}
\PrintBackRefs{\CurrentBib}

\bibitem [\protect \citeauthoryear {%
Gabbay%
}{%
Gabbay%
}{%
{\protect \APACyear {1996}}%
}]{%
dg96}
\APACinsertmetastar {%
dg96}%
\begin{APACrefauthors}%
Gabbay, D.%
\end{APACrefauthors}%
\unskip\
\newblock
\APACrefYear{1996}.
\newblock
(\PrintOrdinal{First}\ \BEd, \BVOL~1).
\newblock
\APACaddressPublisher{}{Clarendon Press}.
\PrintBackRefs{\CurrentBib}

\bibitem [\protect \citeauthoryear {%
Gorren%
}{%
Gorren%
}{%
{\protect \APACyear {1981}}%
}]{%
gj}
\APACinsertmetastar {%
gj}%
\begin{APACrefauthors}%
Gorren, J.%
\end{APACrefauthors}%
\unskip\
\newblock
\APACrefYearMonthDay{1981}{}{}.
\newblock
{\BBOQ}\APACrefatitle {{Theorie Analytique de la Dialectique}} {{Theorie
  Analytique de la Dialectique}}.{\BBCQ}
\newblock
\APACjournalVolNumPages{South West Philosophical Studies}{6}{}{41--47}.
\PrintBackRefs{\CurrentBib}

\bibitem [\protect \citeauthoryear {%
Gr{\"a}tzer%
}{%
Gr{\"a}tzer%
}{%
{\protect \APACyear {1998}}%
}]{%
gg1998}
\APACinsertmetastar {%
gg1998}%
\begin{APACrefauthors}%
Gr{\"a}tzer, G.%
\end{APACrefauthors}%
\unskip\
\newblock
\APACrefYear{1998}.
\newblock
\APACrefbtitle {{General Lattice Theory}} {{General Lattice Theory}}.
\newblock
\APACaddressPublisher{}{Birkhauser}.
\PrintBackRefs{\CurrentBib}

\bibitem [\protect \citeauthoryear {%
Grim%
}{%
Grim%
}{%
{\protect \APACyear {2007}}%
}]{%
gp}
\APACinsertmetastar {%
gp}%
\begin{APACrefauthors}%
Grim, P.%
\end{APACrefauthors}%
\unskip\
\newblock
\APACrefYearMonthDay{2007}{}{}.
\newblock
{\BBOQ}\APACrefatitle {{What is Contradiction ?}} {{What is Contradiction
  ?}}{\BBCQ}
\newblock
\BIn{} G.~Priest\ (\BED), \APACrefbtitle {{The Law of Non-Contradiction}.}
  {{The Law of Non-Contradiction}.}
\newblock
\APACaddressPublisher{}{Oxford Universities Press}.
\PrintBackRefs{\CurrentBib}

\bibitem [\protect \citeauthoryear {%
Hofmann%
}{%
Hofmann%
}{%
{\protect \APACyear {1986}}%
}]{%
hw}
\APACinsertmetastar {%
hw}%
\begin{APACrefauthors}%
Hofmann, W\BPBI C.%
\end{APACrefauthors}%
\unskip\
\newblock
\APACrefYearMonthDay{1986}{}{}.
\newblock
{\BBOQ}\APACrefatitle {{A Formal Model for Dialectical Psychology}} {{A Formal
  Model for Dialectical Psychology}}.{\BBCQ}
\newblock
\APACjournalVolNumPages{International Logic Review}{}{}{40--67}.
\PrintBackRefs{\CurrentBib}

\bibitem [\protect \citeauthoryear {%
Hyde%
\ \BBA {} Colyvan%
}{%
Hyde%
\ \BBA {} Colyvan%
}{%
{\protect \APACyear {2008}}%
}]{%
hc}
\APACinsertmetastar {%
hc}%
\begin{APACrefauthors}%
Hyde, D.%
\BCBT {}\ \BBA {} Colyvan, M.%
\end{APACrefauthors}%
\unskip\
\newblock
\APACrefYearMonthDay{2008}{}{}.
\newblock
{\BBOQ}\APACrefatitle {{Paraconsistent Vagueness-Why Not?}} {{Paraconsistent
  Vagueness-Why Not?}}{\BBCQ}
\newblock
\APACjournalVolNumPages{Australasian J. Logic}{6}{}{207--225}.
\PrintBackRefs{\CurrentBib}

\bibitem [\protect \citeauthoryear {%
Ioan%
}{%
Ioan%
}{%
{\protect \APACyear {1998}}%
}]{%
ip}
\APACinsertmetastar {%
ip}%
\begin{APACrefauthors}%
Ioan, P.%
\end{APACrefauthors}%
\unskip\
\newblock
\APACrefYear{1998}.
\newblock
\APACrefbtitle {{Logic and Dialectics}} {{Logic and Dialectics}}.
\newblock
\APACaddressPublisher{}{Al. I. Cuza Universities Press}.
\PrintBackRefs{\CurrentBib}

\bibitem [\protect \citeauthoryear {%
Iwinski%
}{%
Iwinski%
}{%
{\protect \APACyear {1988}}%
}]{%
it2}
\APACinsertmetastar {%
it2}%
\begin{APACrefauthors}%
Iwinski, T\BPBI B.%
\end{APACrefauthors}%
\unskip\
\newblock
\APACrefYearMonthDay{1988}{}{}.
\newblock
{\BBOQ}\APACrefatitle {{Rough Orders and Rough Concepts}} {{Rough Orders and
  Rough Concepts}}.{\BBCQ}
\newblock
\APACjournalVolNumPages{Bull. Pol. Acad. Sci (Math)}{(3--4)}{}{187--192}.
\PrintBackRefs{\CurrentBib}

\bibitem [\protect \citeauthoryear {%
Koh%
}{%
Koh%
}{%
{\protect \APACyear {1983}}%
}]{%
koh}
\APACinsertmetastar {%
koh}%
\begin{APACrefauthors}%
Koh, K.%
\end{APACrefauthors}%
\unskip\
\newblock
\APACrefYearMonthDay{1983}{}{}.
\newblock
{\BBOQ}\APACrefatitle {{On The Lattice of Maximum-Sized Antichains of A Finite
  Poset}} {{On The Lattice of Maximum-Sized Antichains of A Finite
  Poset}}.{\BBCQ}
\newblock
\APACjournalVolNumPages{Algebra Universalis}{17}{}{73--86}.
\PrintBackRefs{\CurrentBib}

\bibitem [\protect \citeauthoryear {%
Lin%
}{%
Lin%
}{%
{\protect \APACyear {2009}}%
}]{%
tyl}
\APACinsertmetastar {%
tyl}%
\begin{APACrefauthors}%
Lin, T\BPBI Y.%
\end{APACrefauthors}%
\unskip\
\newblock
\APACrefYearMonthDay{2009}{}{}.
\newblock
{\BBOQ}\APACrefatitle {{Granular Computing-1: The Concept of Granulation and
  Its Formal Model}} {{Granular Computing-1: The Concept of Granulation and Its
  Formal Model}}.{\BBCQ}
\newblock
\APACjournalVolNumPages{Int. J. Granular Computing, Rough Sets and Int
  Systems}{1}{1}{21--42}.
\PrintBackRefs{\CurrentBib}

\bibitem [\protect \citeauthoryear {%
Mani%
}{%
Mani%
}{%
{\protect \APACyear {1999}}%
}]{%
am999}
\APACinsertmetastar {%
am999}%
\begin{APACrefauthors}%
Mani, A.%
\end{APACrefauthors}%
\unskip\
\newblock
\APACrefYearMonthDay{1999}{}{}.
\newblock
\APACrefbtitle {{Towards Formal Dialectical logics}} {{Towards Formal
  Dialectical logics}}\ \APACbVolEdTR{}{\BTR{}}.
\PrintBackRefs{\CurrentBib}

\bibitem [\protect \citeauthoryear {%
Mani%
}{%
Mani%
}{%
{\protect \APACyear {2005}}%
}]{%
am3}
\APACinsertmetastar {%
am3}%
\begin{APACrefauthors}%
Mani, A.%
\end{APACrefauthors}%
\unskip\
\newblock
\APACrefYearMonthDay{2005}{}{}.
\newblock
{\BBOQ}\APACrefatitle {{Super Rough Semantics}} {{Super Rough
  Semantics}}.{\BBCQ}
\newblock
\APACjournalVolNumPages{Fundamenta Informaticae}{65}{3}{249--261}.
\PrintBackRefs{\CurrentBib}

\bibitem [\protect \citeauthoryear {%
Mani%
}{%
Mani%
}{%
{\protect \APACyear {2008}}%
}]{%
am24}
\APACinsertmetastar {%
am24}%
\begin{APACrefauthors}%
Mani, A.%
\end{APACrefauthors}%
\unskip\
\newblock
\APACrefYearMonthDay{2008}{}{}.
\newblock
{\BBOQ}\APACrefatitle {{Esoteric Rough Set Theory-Algebraic Semantics of a
  Generalized VPRS and VPRFS}} {{Esoteric Rough Set Theory-Algebraic Semantics
  of a Generalized VPRS and VPRFS}}.{\BBCQ}
\newblock
\BIn{} A.~Skowron\ \BBA {} J\BPBI F.~Peters\ (\BEDS), \APACrefbtitle
  {{Transactions on Rough Sets, LNCS 5084}} {{Transactions on Rough Sets, LNCS
  5084}}\ (\BVOL\ VIII, \BPGS\ 182--231).
\newblock
\APACaddressPublisher{}{Springer Verlag}.
\PrintBackRefs{\CurrentBib}

\bibitem [\protect \citeauthoryear {%
Mani%
}{%
Mani%
}{%
{\protect \APACyear {2009}}%
{\protect \APACexlab {{\protect \BCnt {1}}}}}]{%
am105}
\APACinsertmetastar {%
am105}%
\begin{APACrefauthors}%
Mani, A.%
\end{APACrefauthors}%
\unskip\
\newblock
\APACrefYearMonthDay{2009{\protect \BCnt {1}}}{}{}.
\newblock
{\BBOQ}\APACrefatitle {{Algebraic Semantics of Similarity-Based Bitten Rough
  Set Theory}} {{Algebraic Semantics of Similarity-Based Bitten Rough Set
  Theory}}.{\BBCQ}
\newblock
\APACjournalVolNumPages{Fundamenta Informaticae}{97}{1-2}{177--197}.
\PrintBackRefs{\CurrentBib}

\bibitem [\protect \citeauthoryear {%
Mani%
}{%
Mani%
}{%
{\protect \APACyear {2009}}%
{\protect \APACexlab {{\protect \BCnt {2}}}}}]{%
am699}
\APACinsertmetastar {%
am699}%
\begin{APACrefauthors}%
Mani, A.%
\end{APACrefauthors}%
\unskip\
\newblock
\APACrefYearMonthDay{2009{\protect \BCnt {2}}}{}{}.
\newblock
{\BBOQ}\APACrefatitle {{Integrated Dialectical Logics for Relativised General
  Rough Set Theory}} {{Integrated Dialectical Logics for Relativised General
  Rough Set Theory}}.{\BBCQ}
\newblock
\BIn{} \APACrefbtitle {{Internat. Conf. on Rough Sets, Fuzzy Sets and Soft
  Computing, Agartala, India}} {{Internat. Conf. on Rough Sets, Fuzzy Sets and
  Soft Computing, Agartala, India}}\ (\BPG~6pp (Refereed)).
\newblock
\APACaddressPublisher{}{http://arxiv.org/abs/0909.4876}.
\newblock
\begin{APACrefURL} \url{http://arxiv.org/abs/0909.4876} \end{APACrefURL}
\PrintBackRefs{\CurrentBib}

\bibitem [\protect \citeauthoryear {%
Mani%
}{%
Mani%
}{%
{\protect \APACyear {2011}}%
}]{%
am99}
\APACinsertmetastar {%
am99}%
\begin{APACrefauthors}%
Mani, A.%
\end{APACrefauthors}%
\unskip\
\newblock
\APACrefYearMonthDay{2011}{}{}.
\newblock
{\BBOQ}\APACrefatitle {{Choice Inclusive General Rough Semantics}} {{Choice
  Inclusive General Rough Semantics}}.{\BBCQ}
\newblock
\APACjournalVolNumPages{Information Sciences}{181}{6}{1097--1115}.
\newblock
\begin{APACrefURL} \url{http://dx.doi.org/10.1016/j.ins.2010.11.016}
  \end{APACrefURL}
\newblock
\begin{APACrefDOI} \doi{10.1016/j.ins.2010.11.016} \end{APACrefDOI}
\PrintBackRefs{\CurrentBib}

\bibitem [\protect \citeauthoryear {%
Mani%
}{%
Mani%
}{%
{\protect \APACyear {2012}}%
{\protect \APACexlab {{\protect \BCnt {1}}}}}]{%
am1800}
\APACinsertmetastar {%
am1800}%
\begin{APACrefauthors}%
Mani, A.%
\end{APACrefauthors}%
\unskip\
\newblock
\APACrefYearMonthDay{2012{\protect \BCnt {1}}}{}{}.
\newblock
{\BBOQ}\APACrefatitle {{Axiomatic Approach to Granular Correspondences}}
  {{Axiomatic Approach to Granular Correspondences}}.{\BBCQ}
\newblock
\BIn{} T.~Li\ \BOthers {.}\ (\BEDS), \APACrefbtitle {{Proceedings of
  RSKT'2012}} {{Proceedings of RSKT'2012}}\ (\BVOL\ LNAI 7414, \BPGS\
  482--487).
\newblock
\APACaddressPublisher{}{Springer-Verlag}.
\PrintBackRefs{\CurrentBib}

\bibitem [\protect \citeauthoryear {%
Mani%
}{%
Mani%
}{%
{\protect \APACyear {2012}}%
{\protect \APACexlab {{\protect \BCnt {2}}}}}]{%
am240}
\APACinsertmetastar {%
am240}%
\begin{APACrefauthors}%
Mani, A.%
\end{APACrefauthors}%
\unskip\
\newblock
\APACrefYearMonthDay{2012{\protect \BCnt {2}}}{}{}.
\newblock
{\BBOQ}\APACrefatitle {{Dialectics of Counting and The Mathematics of
  Vagueness}} {{Dialectics of Counting and The Mathematics of
  Vagueness}}.{\BBCQ}
\newblock
\BIn{} J\BPBI F.~Peters\ \BBA {} A.~Skowron\ (\BEDS), \APACrefbtitle
  {{Transactions on Rough Sets , LNCS 7255}} {{Transactions on Rough Sets ,
  LNCS 7255}}\ (\BVOL~XV, \BPGS\ 122--180).
\newblock
\APACaddressPublisher{}{Springer Verlag}.
\PrintBackRefs{\CurrentBib}

\bibitem [\protect \citeauthoryear {%
Mani%
}{%
Mani%
}{%
{\protect \APACyear {2013}}%
}]{%
am909}
\APACinsertmetastar {%
am909}%
\begin{APACrefauthors}%
Mani, A.%
\end{APACrefauthors}%
\unskip\
\newblock
\APACrefYearMonthDay{2013}{}{}.
\newblock
{\BBOQ}\APACrefatitle {{Towards Logics of Some Rough Perspectives of
  Knowledge}} {{Towards Logics of Some Rough Perspectives of
  Knowledge}}.{\BBCQ}
\newblock
\BIn{} Z.~Suraj\ \BBA {} A.~Skowron\ (\BEDS), \APACrefbtitle {{Intelligent
  Systems Reference Library dedicated to the memory of Prof. Pawlak ISRL 43 }}
  {{Intelligent Systems Reference Library dedicated to the memory of Prof.
  Pawlak ISRL 43 }}\ (\BPGS\ 419--444).
\newblock
\APACaddressPublisher{}{Springer Verlag}.
\PrintBackRefs{\CurrentBib}

\bibitem [\protect \citeauthoryear {%
Mani%
}{%
Mani%
}{%
{\protect \APACyear {2013-15}}%
}]{%
am3690}
\APACinsertmetastar {%
am3690}%
\begin{APACrefauthors}%
Mani, A.%
\end{APACrefauthors}%
\unskip\
\newblock
\APACrefYearMonthDay{2013-15}{}{}.
\newblock
{\BBOQ}\APACrefatitle {{Approximation Dialectics of Proto-Transitive Rough
  Sets}} {{Approximation Dialectics of Proto-Transitive Rough Sets}}.{\BBCQ}
\newblock
\BIn{} M\BPBI K.~Chakraborty, A.~Skowron\BCBL {}\ \BBA {} S.~Kar\ (\BEDS),
  \APACrefbtitle {{Facets of Uncertainties and Applications}} {{Facets of
  Uncertainties and Applications}}\ (\BPGS\ 99--109).
\newblock
\APACaddressPublisher{}{Springer}.
\PrintBackRefs{\CurrentBib}

\bibitem [\protect \citeauthoryear {%
Mani%
}{%
Mani%
}{%
{\protect \APACyear {2014}}%
{\protect \APACexlab {{\protect \BCnt {1}}}}}]{%
am6000}
\APACinsertmetastar {%
am6000}%
\begin{APACrefauthors}%
Mani, A.%
\end{APACrefauthors}%
\unskip\
\newblock
\APACrefYear{2014{\protect \BCnt {1}}}.
\newblock
\APACrefbtitle {{Algebraic Semantics of Proto-Transitive Rough Sets}}
  {{Algebraic Semantics of Proto-Transitive Rough Sets}}\
  (\PrintOrdinal{First}\ \BEd).
\newblock
\APACaddressPublisher{}{Arxiv:1410.0572}.
\newblock
\begin{APACrefURL} \url{http://arxiv.org/abs/1410.0572} \end{APACrefURL}
\PrintBackRefs{\CurrentBib}

\bibitem [\protect \citeauthoryear {%
Mani%
}{%
Mani%
}{%
{\protect \APACyear {2014}}%
{\protect \APACexlab {{\protect \BCnt {2}}}}}]{%
am3930}
\APACinsertmetastar {%
am3930}%
\begin{APACrefauthors}%
Mani, A.%
\end{APACrefauthors}%
\unskip\
\newblock
\APACrefYearMonthDay{2014{\protect \BCnt {2}}}{}{}.
\newblock
{\BBOQ}\APACrefatitle {{Ontology, Rough Y-Systems and Dependence}} {{Ontology,
  Rough Y-Systems and Dependence}}.{\BBCQ}
\newblock
\APACjournalVolNumPages{Internat. J of Comp. Sci. and Appl.}{11}{2}{114--136}.
\newblock
\APACrefnote{Special Issue of IJCSA on Computational Intelligence}
\PrintBackRefs{\CurrentBib}

\bibitem [\protect \citeauthoryear {%
Mani%
}{%
Mani%
}{%
{\protect \APACyear {2015}}%
}]{%
am6999}
\APACinsertmetastar {%
am6999}%
\begin{APACrefauthors}%
Mani, A.%
\end{APACrefauthors}%
\unskip\
\newblock
\APACrefYearMonthDay{2015}{}{}.
\newblock
{\BBOQ}\APACrefatitle {{Antichain Based Semantics for Rough Sets}} {{Antichain
  Based Semantics for Rough Sets}}.{\BBCQ}
\newblock
\BIn{} D.~Ciucci, G.~Wang, S.~Mitra\BCBL {}\ \BBA {} W.~Wu\ (\BEDS),
  \APACrefbtitle {{RSKT 2015}} {{RSKT 2015}}\ (\BPGS\ 319--330).
\newblock
\APACaddressPublisher{}{Springer-Verlag}.
\PrintBackRefs{\CurrentBib}

\bibitem [\protect \citeauthoryear {%
Mani%
}{%
Mani%
}{%
{\protect \APACyear {2016}}%
{\protect \APACexlab {{\protect \BCnt {1}}}}}]{%
am9501}
\APACinsertmetastar {%
am9501}%
\begin{APACrefauthors}%
Mani, A.%
\end{APACrefauthors}%
\unskip\
\newblock
\APACrefYearMonthDay{2016{\protect \BCnt {1}}}{}{}.
\newblock
{\BBOQ}\APACrefatitle {{Algebraic Semantics of Proto-Transitive Rough Sets}}
  {{Algebraic Semantics of Proto-Transitive Rough Sets}}.{\BBCQ}
\newblock
\BIn{} J\BPBI F.~Peters\ \BBA {} A.~Skowron\ (\BEDS), \APACrefbtitle
  {{Transactions on Rough Sets LNCS 10020}} {{Transactions on Rough Sets LNCS
  10020}}\ (\BVOL~XX, \BPGS\ 51--108).
\newblock
\APACaddressPublisher{}{Springer Verlag}.
\PrintBackRefs{\CurrentBib}

\bibitem [\protect \citeauthoryear {%
Mani%
}{%
Mani%
}{%
{\protect \APACyear {2016}}%
{\protect \APACexlab {{\protect \BCnt {2}}}}}]{%
am9699}
\APACinsertmetastar {%
am9699}%
\begin{APACrefauthors}%
Mani, A.%
\end{APACrefauthors}%
\unskip\
\newblock
\APACrefYearMonthDay{2016{\protect \BCnt {2}}}{October}{}.
\newblock
{\BBOQ}\APACrefatitle {{On Deductive Systems of AC Semantics for Rough Sets}}
  {{On Deductive Systems of AC Semantics for Rough Sets}}.{\BBCQ}
\newblock
\APACjournalVolNumPages{ArXiv. Math}{}{1610.02634v1}{1--12}.
\PrintBackRefs{\CurrentBib}

\bibitem [\protect \citeauthoryear {%
Mani%
}{%
Mani%
}{%
{\protect \APACyear {2016}}%
{\protect \APACexlab {{\protect \BCnt {3}}}}}]{%
am6900}
\APACinsertmetastar {%
am6900}%
\begin{APACrefauthors}%
Mani, A.%
\end{APACrefauthors}%
\unskip\
\newblock
\APACrefYearMonthDay{2016{\protect \BCnt {3}}}{}{}.
\newblock
{\BBOQ}\APACrefatitle {{Pure Rough Mereology and Counting}} {{Pure Rough
  Mereology and Counting}}.{\BBCQ}
\newblock
\BIn{} \APACrefbtitle {{WIECON,2016}} {{WIECON,2016}}\ (\BPGS\ 1--8).
\newblock
\APACaddressPublisher{}{IEEXPlore}.
\PrintBackRefs{\CurrentBib}

\bibitem [\protect \citeauthoryear {%
Mani%
}{%
Mani%
}{%
{\protect \APACyear {2017}}%
{\protect \APACexlab {{\protect \BCnt {1}}}}}]{%
am9006}
\APACinsertmetastar {%
am9006}%
\begin{APACrefauthors}%
Mani, A.%
\end{APACrefauthors}%
\unskip\
\newblock
\APACrefYearMonthDay{2017{\protect \BCnt {1}}}{}{}.
\newblock
{\BBOQ}\APACrefatitle {{Approximations From Anywhere and General Rough Sets}}
  {{Approximations From Anywhere and General Rough Sets}}.{\BBCQ}
\newblock
\BIn{} L.~Polkowski\ \BOthers {.}\ (\BEDS), \APACrefbtitle {{Rough Sets-2,
  IJCRS,2017}} {{Rough Sets-2, IJCRS,2017}}\ (\BPGS\ 3--22).
\newblock
\APACaddressPublisher{}{Springer International}.
\newblock
\begin{APACrefDOI} \doi{10.1007/978-3-319-60840-2} \end{APACrefDOI}
\PrintBackRefs{\CurrentBib}

\bibitem [\protect \citeauthoryear {%
Mani%
}{%
Mani%
}{%
{\protect \APACyear {2017}}%
{\protect \APACexlab {{\protect \BCnt {2}}}}}]{%
am9114}
\APACinsertmetastar {%
am9114}%
\begin{APACrefauthors}%
Mani, A.%
\end{APACrefauthors}%
\unskip\
\newblock
\APACrefYearMonthDay{2017{\protect \BCnt {2}}}{}{}.
\newblock
{\BBOQ}\APACrefatitle {{Knowledge and Consequence in AC Semantics for General
  Rough Sets }} {{Knowledge and Consequence in AC Semantics for General Rough
  Sets }}.{\BBCQ}
\newblock
\BIn{} G.~Wang, A.~Skowron\BCBL {}\ \BBA {} D.~Yao Y. Y.and~{\'S}l{\c e}zak\
  (\BEDS), \APACrefbtitle {{Thriving Rough Sets----10th Anniversary - Honoring
  Prof Pawlak \& 35 years of Rough Sets,}} {{Thriving Rough Sets----10th
  Anniversary - Honoring Prof Pawlak \& 35 years of Rough Sets,}}\ (\BVOL~708,
  \BPGS\ 237--268).
\newblock
\APACaddressPublisher{}{Springer International Publishing}.
\newblock
\begin{APACrefDOI} \doi{10.1007/978-3-319-54966-8} \end{APACrefDOI}
\PrintBackRefs{\CurrentBib}

\bibitem [\protect \citeauthoryear {%
Mani%
}{%
Mani%
}{%
{\protect \APACyear {2018}}%
{\protect \APACexlab {{\protect \BCnt {1}}}}}]{%
am501}
\APACinsertmetastar {%
am501}%
\begin{APACrefauthors}%
Mani, A.%
\end{APACrefauthors}%
\unskip\
\newblock
\APACrefYearMonthDay{2018{\protect \BCnt {1}}}{}{}.
\newblock
{\BBOQ}\APACrefatitle {{Algebraic Methods for Granular Rough Sets}} {{Algebraic
  Methods for Granular Rough Sets}}.{\BBCQ}
\newblock
\BIn{} A.~Mani, I.~D{\"u}ntsch\BCBL {}\ \BBA {} G.~Cattaneo\ (\BEDS),
  \APACrefbtitle {{Algebraic Methods in General Rough Sets}} {{Algebraic
  Methods in General Rough Sets}}\ (\BPGS\ 125--305).
\newblock
\APACaddressPublisher{}{Springer International}.
\PrintBackRefs{\CurrentBib}

\bibitem [\protect \citeauthoryear {%
Mani%
}{%
Mani%
}{%
{\protect \APACyear {2018}}%
{\protect \APACexlab {{\protect \BCnt {2}}}}}]{%
am5019}
\APACinsertmetastar {%
am5019}%
\begin{APACrefauthors}%
Mani, A.%
\end{APACrefauthors}%
\unskip\
\newblock
\APACrefYearMonthDay{2018{\protect \BCnt {2}}}{}{}.
\newblock
{\BBOQ}\APACrefatitle {{Representation, Duality and Beyond}} {{Representation,
  Duality and Beyond}}.{\BBCQ}
\newblock
\BIn{} A.~Mani, I.~D{\"u}ntsch\BCBL {}\ \BBA {} G.~Cattaneo\ (\BEDS),
  \APACrefbtitle {{Algebraic Methods in General Rough Sets}} {{Algebraic
  Methods in General Rough Sets}}\ (\BPGS\ 425--533).
\newblock
\APACaddressPublisher{}{Springer International}.
\PrintBackRefs{\CurrentBib}

\bibitem [\protect \citeauthoryear {%
Marx%
\ \BBA {} Engels%
}{%
Marx%
\ \BBA {} Engels%
}{%
{\protect \APACyear {1989}}%
}]{%
mecw24}
\APACinsertmetastar {%
mecw24}%
\begin{APACrefauthors}%
Marx, K.%
\BCBT {}\ \BBA {} Engels, F.%
\end{APACrefauthors}%
\unskip\
\newblock
\APACrefYear{1989}.
\newblock
\APACrefbtitle {{Marx and Engels: Collected Works: Vol 24}} {{Marx and Engels:
  Collected Works: Vol 24}}.
\newblock
\APACaddressPublisher{}{Progress Publishers}.
\PrintBackRefs{\CurrentBib}

\bibitem [\protect \citeauthoryear {%
McGill%
\ \BBA {} Parry%
}{%
McGill%
\ \BBA {} Parry%
}{%
{\protect \APACyear {1948}}%
}]{%
mvp}
\APACinsertmetastar {%
mvp}%
\begin{APACrefauthors}%
McGill, V\BPBI P.%
\BCBT {}\ \BBA {} Parry, W\BPBI T.%
\end{APACrefauthors}%
\unskip\
\newblock
\APACrefYearMonthDay{1948}{}{}.
\newblock
{\BBOQ}\APACrefatitle {{The Unity of Opposites - A Dialectical Principle}}
  {{The Unity of Opposites - A Dialectical Principle}}.{\BBCQ}
\newblock
\APACjournalVolNumPages{Science and Society}{12}{}{418--444}.
\PrintBackRefs{\CurrentBib}

\bibitem [\protect \citeauthoryear {%
Moore%
\ \BBA {} Shannon%
}{%
Moore%
\ \BBA {} Shannon%
}{%
{\protect \APACyear {1956}}%
}]{%
sha56}
\APACinsertmetastar {%
sha56}%
\begin{APACrefauthors}%
Moore, E\BPBI F.%
\BCBT {}\ \BBA {} Shannon, C\BPBI E.%
\end{APACrefauthors}%
\unskip\
\newblock
\APACrefYearMonthDay{1956}{}{}.
\newblock
{\BBOQ}\APACrefatitle {{Reliable Circuits Using Less Reliable Relays-I, II}}
  {{Reliable Circuits Using Less Reliable Relays-I, II}}.{\BBCQ}
\newblock
\APACjournalVolNumPages{Bell Systems Technical Journal}{}{}{191--208,
  281--297}.
\PrintBackRefs{\CurrentBib}

\bibitem [\protect \citeauthoryear {%
Moretti%
}{%
Moretti%
}{%
{\protect \APACyear {2012}}%
}]{%
mo2}
\APACinsertmetastar {%
mo2}%
\begin{APACrefauthors}%
Moretti, A.%
\end{APACrefauthors}%
\unskip\
\newblock
\APACrefYearMonthDay{2012}{}{}.
\newblock
{\BBOQ}\APACrefatitle {{Why The Logical Hexagon?}} {{Why The Logical
  Hexagon?}}{\BBCQ}
\newblock
\APACjournalVolNumPages{Logica Univers}{6}{}{69--107}.
\PrintBackRefs{\CurrentBib}

\bibitem [\protect \citeauthoryear {%
Pagliani%
}{%
Pagliani%
}{%
{\protect \APACyear {1998}}%
}]{%
pp1998}
\APACinsertmetastar {%
pp1998}%
\begin{APACrefauthors}%
Pagliani, P.%
\end{APACrefauthors}%
\unskip\
\newblock
\APACrefYearMonthDay{1998}{}{}.
\newblock
{\BBOQ}\APACrefatitle {{Rough Set Theory and Logico-algebraic Structures}}
  {{Rough Set Theory and Logico-algebraic Structures}}.{\BBCQ}
\newblock
\BIn{} E.~Or{\l}owska\ (\BED), \APACrefbtitle {{Incomplete Information: Rough
  Set Analysis}} {{Incomplete Information: Rough Set Analysis}}\ (\BPGS\
  109--190).
\newblock
\APACaddressPublisher{}{Physica Verlag}.
\PrintBackRefs{\CurrentBib}

\bibitem [\protect \citeauthoryear {%
Pagliani%
}{%
Pagliani%
}{%
{\protect \APACyear {2000}}%
}]{%
pp2000}
\APACinsertmetastar {%
pp2000}%
\begin{APACrefauthors}%
Pagliani, P.%
\end{APACrefauthors}%
\unskip\
\newblock
\APACrefYearMonthDay{2000}{}{}.
\newblock
{\BBOQ}\APACrefatitle {{Local Classical Behaviours in Three-Valued Logics and
  Connected Systems. Part 1}} {{Local Classical Behaviours in Three-Valued
  Logics and Connected Systems. Part 1}}.{\BBCQ}
\newblock
\APACjournalVolNumPages{Journal of Multiple valued Logics}{5}{}{327--347}.
\PrintBackRefs{\CurrentBib}

\bibitem [\protect \citeauthoryear {%
Pagliani%
}{%
Pagliani%
}{%
{\protect \APACyear {2016}}%
}]{%
pp20}
\APACinsertmetastar {%
pp20}%
\begin{APACrefauthors}%
Pagliani, P.%
\end{APACrefauthors}%
\unskip\
\newblock
\APACrefYearMonthDay{2016}{}{}.
\newblock
{\BBOQ}\APACrefatitle {{Covering Rough Sets and Formal Topology -- A Uniform
  Approach Through Intensional and Extensional Constructors }} {{Covering Rough
  Sets and Formal Topology -- A Uniform Approach Through Intensional and
  Extensional Constructors }}.{\BBCQ}
\newblock
\BIn{} J\BPBI F.~Peters\ \BBA {} A.~Skowron\ (\BEDS), \APACrefbtitle
  {{Transactions on Rough Sets LNCS 10020}} {{Transactions on Rough Sets LNCS
  10020}}\ (\BVOL~XX, \BPGS\ 109--145).
\newblock
\APACaddressPublisher{}{Springer-Verlag}.
\newblock
\begin{APACrefDOI} \doi{10.1007/978-3-662-53611-7} \end{APACrefDOI}
\PrintBackRefs{\CurrentBib}

\bibitem [\protect \citeauthoryear {%
Pagliani%
\ \BBA {} Chakraborty%
}{%
Pagliani%
\ \BBA {} Chakraborty%
}{%
{\protect \APACyear {2008}}%
}]{%
ppm2}
\APACinsertmetastar {%
ppm2}%
\begin{APACrefauthors}%
Pagliani, P.%
\BCBT {}\ \BBA {} Chakraborty, M.%
\end{APACrefauthors}%
\unskip\
\newblock
\APACrefYear{2008}.
\newblock
\APACrefbtitle {{A Geometry of Approximation: Rough Set Theory: Logic, Algebra
  and Topology of Conceptual Patterns}} {{A Geometry of Approximation: Rough
  Set Theory: Logic, Algebra and Topology of Conceptual Patterns}}.
\newblock
\APACaddressPublisher{Berlin}{Springer}.
\PrintBackRefs{\CurrentBib}

\bibitem [\protect \citeauthoryear {%
Pawlak%
}{%
Pawlak%
}{%
{\protect \APACyear {1987}}%
}]{%
zp0}
\APACinsertmetastar {%
zp0}%
\begin{APACrefauthors}%
Pawlak, Z.%
\end{APACrefauthors}%
\unskip\
\newblock
\APACrefYearMonthDay{1987}{}{}.
\newblock
{\BBOQ}\APACrefatitle {{Rough Logic}} {{Rough Logic}}.{\BBCQ}
\newblock
\APACjournalVolNumPages{Bull.Pol.Acad.Sci (Tech)}{35}{}{253--258}.
\PrintBackRefs{\CurrentBib}

\bibitem [\protect \citeauthoryear {%
Pawlak%
}{%
Pawlak%
}{%
{\protect \APACyear {1991}}%
}]{%
zpb}
\APACinsertmetastar {%
zpb}%
\begin{APACrefauthors}%
Pawlak, Z.%
\end{APACrefauthors}%
\unskip\
\newblock
\APACrefYear{1991}.
\newblock
\APACrefbtitle {{Rough Sets: Theoretical Aspects of Reasoning About Data}}
  {{Rough Sets: Theoretical Aspects of Reasoning About Data}}.
\newblock
\APACaddressPublisher{Dodrecht}{Kluwer Academic Publishers}.
\PrintBackRefs{\CurrentBib}

\bibitem [\protect \citeauthoryear {%
Polkowski%
}{%
Polkowski%
}{%
{\protect \APACyear {2011}}%
}]{%
lp2011}
\APACinsertmetastar {%
lp2011}%
\begin{APACrefauthors}%
Polkowski, L.%
\end{APACrefauthors}%
\unskip\
\newblock
\APACrefYear{2011}.
\newblock
\APACrefbtitle {{Approximate Reasoning by Parts}} {{Approximate Reasoning by
  Parts}}.
\newblock
\APACaddressPublisher{}{Springer Verlag}.
\PrintBackRefs{\CurrentBib}

\bibitem [\protect \citeauthoryear {%
Polkowski%
\ \BBA {} Skowron%
}{%
Polkowski%
\ \BBA {} Skowron%
}{%
{\protect \APACyear {1996}}%
}]{%
ps3}
\APACinsertmetastar {%
ps3}%
\begin{APACrefauthors}%
Polkowski, L.%
\BCBT {}\ \BBA {} Skowron, A.%
\end{APACrefauthors}%
\unskip\
\newblock
\APACrefYearMonthDay{1996}{}{}.
\newblock
{\BBOQ}\APACrefatitle {{Rough Mereology: A New Paradigm for Approximate
  Reasoning}} {{Rough Mereology: A New Paradigm for Approximate
  Reasoning}}.{\BBCQ}
\newblock
\APACjournalVolNumPages{Internat. J. Appr. Reasoning}{15}{4}{333--365}.
\PrintBackRefs{\CurrentBib}

\bibitem [\protect \citeauthoryear {%
Priest%
}{%
Priest%
}{%
{\protect \APACyear {1984}}%
}]{%
gp5}
\APACinsertmetastar {%
gp5}%
\begin{APACrefauthors}%
Priest, G.%
\end{APACrefauthors}%
\unskip\
\newblock
\APACrefYearMonthDay{1984}{}{}.
\newblock
{\BBOQ}\APACrefatitle {{To Be and Not To Be: Dialectical Tense Logic}} {{To Be
  and Not To Be: Dialectical Tense Logic}}.{\BBCQ}
\newblock
\APACjournalVolNumPages{Studia Logica}{41}{2/3}{249--268}.
\PrintBackRefs{\CurrentBib}

\bibitem [\protect \citeauthoryear {%
Priest%
}{%
Priest%
}{%
{\protect \APACyear {1990}}%
}]{%
gp1990}
\APACinsertmetastar {%
gp1990}%
\begin{APACrefauthors}%
Priest, G.%
\end{APACrefauthors}%
\unskip\
\newblock
\APACrefYearMonthDay{1990}{}{}.
\newblock
{\BBOQ}\APACrefatitle {{Dialectic and Dialetheic}} {{Dialectic and
  Dialetheic}}.{\BBCQ}
\newblock
\APACjournalVolNumPages{Science and Society}{53}{4}{388--415}.
\PrintBackRefs{\CurrentBib}

\bibitem [\protect \citeauthoryear {%
Priest%
}{%
Priest%
}{%
{\protect \APACyear {1999}}%
}]{%
gp2}
\APACinsertmetastar {%
gp2}%
\begin{APACrefauthors}%
Priest, G.%
\end{APACrefauthors}%
\unskip\
\newblock
\APACrefYearMonthDay{1999}{}{}.
\newblock
{\BBOQ}\APACrefatitle {{What Not? A Defence of a Dialetheic Theory of
  Negation}} {{What Not? A Defence of a Dialetheic Theory of Negation}}.{\BBCQ}
\newblock
\BIn{} D.~Gabbay\ (\BED), \APACrefbtitle {{What is Negation ?}} {{What is
  Negation ?}}\ (\BPGS\ 101--120).
\newblock
\APACaddressPublisher{}{Kluwer}.
\PrintBackRefs{\CurrentBib}

\bibitem [\protect \citeauthoryear {%
Priest%
}{%
Priest%
}{%
{\protect \APACyear {2006}}%
}]{%
gp2006}
\APACinsertmetastar {%
gp2006}%
\begin{APACrefauthors}%
Priest, G.%
\end{APACrefauthors}%
\unskip\
\newblock
\APACrefYear{2006}.
\newblock
\APACaddressPublisher{}{Oxford University Press}.
\PrintBackRefs{\CurrentBib}

\bibitem [\protect \citeauthoryear {%
Priest%
}{%
Priest%
}{%
{\protect \APACyear {2007}}%
}]{%
gp1}
\APACinsertmetastar {%
gp1}%
\begin{APACrefauthors}%
Priest, G.%
\end{APACrefauthors}%
\unskip\
\newblock
\APACrefYearMonthDay{2007}{}{}.
\newblock
{\BBOQ}\APACrefatitle {{Logicians Setting Together Contradictories - A
  Perspective on Relevance, Paraconsistency and Dialetheism}} {{Logicians
  Setting Together Contradictories - A Perspective on Relevance,
  Paraconsistency and Dialetheism}}.{\BBCQ}
\newblock
\BIn{} D.~Jacquette\ (\BED), \APACrefbtitle {{Blackwell Handbook to
  Philosophical Logic}.} {{Blackwell Handbook to Philosophical Logic}.}
\newblock
\APACaddressPublisher{}{Blackwell}.
\PrintBackRefs{\CurrentBib}

\bibitem [\protect \citeauthoryear {%
Priest%
}{%
Priest%
}{%
{\protect \APACyear {2014}}%
}]{%
gp2014}
\APACinsertmetastar {%
gp2014}%
\begin{APACrefauthors}%
Priest, G.%
\end{APACrefauthors}%
\unskip\
\newblock
\APACrefYearMonthDay{2014}{}{}.
\newblock
{\BBOQ}\APACrefatitle {{Contradictory Concepts}} {{Contradictory
  Concepts}}.{\BBCQ}
\newblock
\BIn{} E.~Weber, D.~Wouters\BCBL {}\ \BBA {} J.~Meheus\ (\BEDS), \APACrefbtitle
  {{Logic, Reasoning and Rationality}} {{Logic, Reasoning and Rationality}}\
  (\BVOL~5, \BPGS\ 197--216).
\newblock
\APACaddressPublisher{}{Springer}.
\PrintBackRefs{\CurrentBib}

\bibitem [\protect \citeauthoryear {%
Rasiowa%
}{%
Rasiowa%
}{%
{\protect \APACyear {1974}}%
}]{%
rh}
\APACinsertmetastar {%
rh}%
\begin{APACrefauthors}%
Rasiowa, H.%
\end{APACrefauthors}%
\unskip\
\newblock
\APACrefYear{1974}.
\newblock
\APACrefbtitle {{An Algebraic Approach to Nonclassical Logics}} {{An Algebraic
  Approach to Nonclassical Logics}}\ (\BVOL~78).
\newblock
\APACaddressPublisher{Warsaw}{North Holland}.
\PrintBackRefs{\CurrentBib}

\bibitem [\protect \citeauthoryear {%
Saha%
, Sen%
\BCBL {}\ \BBA {} Chakraborty%
}{%
Saha%
\ \protect \BOthers {.}}{%
{\protect \APACyear {2015}}%
}]{%
ajm2015}
\APACinsertmetastar {%
ajm2015}%
\begin{APACrefauthors}%
Saha, A.%
, Sen, J.%
\BCBL {}\ \BBA {} Chakraborty, M\BPBI K.%
\end{APACrefauthors}%
\unskip\
\newblock
\APACrefYearMonthDay{2015}{}{}.
\newblock
{\BBOQ}\APACrefatitle {{Algebraic Structures in The Vicinity of Pre-Rough
  Algebra and Their Logics II}} {{Algebraic Structures in The Vicinity of
  Pre-Rough Algebra and Their Logics II}}.{\BBCQ}
\newblock
\APACjournalVolNumPages{Information Sciences}{333}{}{44--60}.
\PrintBackRefs{\CurrentBib}

\bibitem [\protect \citeauthoryear {%
Sambin%
}{%
Sambin%
}{%
{\protect \APACyear {1987}}%
}]{%
gs1987}
\APACinsertmetastar {%
gs1987}%
\begin{APACrefauthors}%
Sambin, G.%
\end{APACrefauthors}%
\unskip\
\newblock
\APACrefYearMonthDay{1987}{}{}.
\newblock
{\BBOQ}\APACrefatitle {{Intuitionistic Formal Spaces - A First Communication}}
  {{Intuitionistic Formal Spaces - A First Communication}}.{\BBCQ}
\newblock
\BIn{} D.~Skordev\ (\BED), \APACrefbtitle {{Mathematical Logic and Its
  Applications}} {{Mathematical Logic and Its Applications}}\ (\BPGS\
  187--204).
\newblock
\APACaddressPublisher{}{Plenum Press}.
\newblock
\begin{APACrefURL} \url{http://www.math.unipd.it/~sambin/txt/ifs87-97.pdf}
  \end{APACrefURL}
\PrintBackRefs{\CurrentBib}

\bibitem [\protect \citeauthoryear {%
Sambin%
\ \BBA {} Gebellato%
}{%
Sambin%
\ \BBA {} Gebellato%
}{%
{\protect \APACyear {1999}}%
}]{%
gs1999}
\APACinsertmetastar {%
gs1999}%
\begin{APACrefauthors}%
Sambin, G.%
\BCBT {}\ \BBA {} Gebellato, S.%
\end{APACrefauthors}%
\unskip\
\newblock
\APACrefYearMonthDay{1999}{}{}.
\newblock
{\BBOQ}\APACrefatitle {{A Preview of The Basic Picture: A New Perspective on
  Formal Topology}} {{A Preview of The Basic Picture: A New Perspective on
  Formal Topology}}.{\BBCQ}
\newblock
\BIn{} T.~Altenkirch, B.~Reus\BCBL {}\ \BBA {} W.~Naraschewski\ (\BEDS),
  \APACrefbtitle {{Types1998}} {{Types1998}}\ (\BPGS\ 194--208).
\newblock
\APACaddressPublisher{}{Springer}.
\newblock
\begin{APACrefDOI} \doi{10.1007/3-540-48167-2} \end{APACrefDOI}
\PrintBackRefs{\CurrentBib}

\bibitem [\protect \citeauthoryear {%
Schang%
}{%
Schang%
}{%
{\protect \APACyear {2012}}%
}]{%
fs2012}
\APACinsertmetastar {%
fs2012}%
\begin{APACrefauthors}%
Schang, F.%
\end{APACrefauthors}%
\unskip\
\newblock
\APACrefYearMonthDay{2012}{}{}.
\newblock
{\BBOQ}\APACrefatitle {{Opposites and oppositions Around and Beyond The Square
  of Opposition}} {{Opposites and oppositions Around and Beyond The Square of
  Opposition}}.{\BBCQ}
\newblock
\BIn{} J\BPBI Y.~Beziau, D.~Jacquette\BCBL {}\ \BOthers {.}\ (\BEDS),
  \APACrefbtitle {{Around and Beyond The Square of Opposition }} {{Around and
  Beyond The Square of Opposition }}\ (\BVOL~I, \BPGS\ 147--174).
\newblock
\APACaddressPublisher{}{Birkhauser}.
\PrintBackRefs{\CurrentBib}

\bibitem [\protect \citeauthoryear {%
Shannon%
}{%
Shannon%
}{%
{\protect \APACyear {1948}}%
}]{%
sha48}
\APACinsertmetastar {%
sha48}%
\begin{APACrefauthors}%
Shannon, C\BPBI E.%
\end{APACrefauthors}%
\unskip\
\newblock
\APACrefYearMonthDay{1948}{}{}.
\newblock
{\BBOQ}\APACrefatitle {{A Mathematical Theory of Communication}} {{A
  Mathematical Theory of Communication}}.{\BBCQ}
\newblock
\APACjournalVolNumPages{Bell Systems Technical Journal}{27}{}{379--423,
  623--656}.
\PrintBackRefs{\CurrentBib}

\bibitem [\protect \citeauthoryear {%
{\'S}l{\c e}zak%
\ \BBA {} Wasilewski%
}{%
{\'S}l{\c e}zak%
\ \BBA {} Wasilewski%
}{%
{\protect \APACyear {2007}}%
}]{%
sw}
\APACinsertmetastar {%
sw}%
\begin{APACrefauthors}%
{\'S}l{\c e}zak, D.%
\BCBT {}\ \BBA {} Wasilewski, P.%
\end{APACrefauthors}%
\unskip\
\newblock
\APACrefYearMonthDay{2007}{}{}.
\newblock
{\BBOQ}\APACrefatitle {{Granular Sets - Foundations and Case Study of Tolerance
  Spaces}} {{Granular Sets - Foundations and Case Study of Tolerance
  Spaces}}.{\BBCQ}
\newblock
\BIn{} A.~An, J.~Stefanowski, S.~Ramanna, C\BPBI J.~Butz, W.~Pedrycz\BCBL {}\
  \BBA {} G.~Wang\ (\BEDS), \APACrefbtitle {{RSFDGrC 2007, LNCS}} {{RSFDGrC
  2007, LNCS}}\ (\BVOL\ 4482, \BPGS\ 435--442).
\newblock
\APACaddressPublisher{}{Springer}.
\PrintBackRefs{\CurrentBib}

\bibitem [\protect \citeauthoryear {%
Swaminathan%
\ \BBA {} Rawal%
}{%
Swaminathan%
\ \BBA {} Rawal%
}{%
{\protect \APACyear {2015}}%
}]{%
ms2015}
\APACinsertmetastar {%
ms2015}%
\begin{APACrefauthors}%
Swaminathan, M.%
\BCBT {}\ \BBA {} Rawal, V.%
\end{APACrefauthors}%
\unskip\
\newblock
\APACrefYear{2015}.
\newblock
\APACaddressPublisher{}{Tulika Books}.
\PrintBackRefs{\CurrentBib}

\bibitem [\protect \citeauthoryear {%
Tzouvaras%
}{%
Tzouvaras%
}{%
{\protect \APACyear {2001}}%
}]{%
at}
\APACinsertmetastar {%
at}%
\begin{APACrefauthors}%
Tzouvaras, A.%
\end{APACrefauthors}%
\unskip\
\newblock
\APACrefYearMonthDay{2001}{}{}.
\newblock
{\BBOQ}\APACrefatitle {{Periodicity of Negation}} {{Periodicity of
  Negation}}.{\BBCQ}
\newblock
\APACjournalVolNumPages{Notre Dame J. Formal Logic}{42}{2}{88--99}.
\PrintBackRefs{\CurrentBib}

\bibitem [\protect \citeauthoryear {%
Woods%
}{%
Woods%
}{%
{\protect \APACyear {2004}}%
}]{%
wj}
\APACinsertmetastar {%
wj}%
\begin{APACrefauthors}%
Woods, J.%
\end{APACrefauthors}%
\unskip\
\newblock
\APACrefYearMonthDay{2004}{Preprint}{}.
\newblock
\APACrefbtitle {{Dialectical Considerations on The Logic of Contradiction: Part
  I, II}.} {{Dialectical Considerations on The Logic of Contradiction: Part I,
  II}.}
\PrintBackRefs{\CurrentBib}

\bibitem [\protect \citeauthoryear {%
Yao%
}{%
Yao%
}{%
{\protect \APACyear {1998}}%
}]{%
yy9}
\APACinsertmetastar {%
yy9}%
\begin{APACrefauthors}%
Yao, Y\BPBI Y.%
\end{APACrefauthors}%
\unskip\
\newblock
\APACrefYearMonthDay{1998}{}{}.
\newblock
{\BBOQ}\APACrefatitle {{Relational Interpretation of Neighbourhood Operators
  and Rough Set Approximation Operators}} {{Relational Interpretation of
  Neighbourhood Operators and Rough Set Approximation Operators}}.{\BBCQ}
\newblock
\APACjournalVolNumPages{Information Sciences}{}{}{239--259}.
\PrintBackRefs{\CurrentBib}

\bibitem [\protect \citeauthoryear {%
Yao%
\ \BBA {} Yao%
}{%
Yao%
\ \BBA {} Yao%
}{%
{\protect \APACyear {2012}}%
}]{%
yy2012c}
\APACinsertmetastar {%
yy2012c}%
\begin{APACrefauthors}%
Yao, Y\BPBI Y.%
\BCBT {}\ \BBA {} Yao, B.%
\end{APACrefauthors}%
\unskip\
\newblock
\APACrefYearMonthDay{2012}{}{}.
\newblock
{\BBOQ}\APACrefatitle {{Covering Based Rough Set Approximations}} {{Covering
  Based Rough Set Approximations}}.{\BBCQ}
\newblock
\APACjournalVolNumPages{Information Sciences}{200}{}{91--107}.
\PrintBackRefs{\CurrentBib}

\bibitem [\protect \citeauthoryear {%
Zeleny%
}{%
Zeleny%
}{%
{\protect \APACyear {1994}}%
}]{%
zj1}
\APACinsertmetastar {%
zj1}%
\begin{APACrefauthors}%
Zeleny, J.%
\end{APACrefauthors}%
\unskip\
\newblock
\APACrefYearMonthDay{1994}{}{}.
\newblock
{\BBOQ}\APACrefatitle {{Paraconsistency and Dialectical Consistency}}
  {{Paraconsistency and Dialectical Consistency}}.{\BBCQ}
\newblock
\APACjournalVolNumPages{From The Logical Point of View}{1}{}{35--51}.
\PrintBackRefs{\CurrentBib}

\bibitem [\protect \citeauthoryear {%
Zimmerman%
}{%
Zimmerman%
}{%
{\protect \APACyear {2000}}%
}]{%
bz2000}
\APACinsertmetastar {%
bz2000}%
\begin{APACrefauthors}%
Zimmerman, B.%
\end{APACrefauthors}%
\ (\BED).
\unskip\
\newblock
\APACrefYear{2000}.
\newblock
\APACrefbtitle {{Lesbian Histories and Cultures: An Encyclopedia }} {{Lesbian
  Histories and Cultures: An Encyclopedia }}.
\newblock
\APACaddressPublisher{}{Garland Publishers}.
\PrintBackRefs{\CurrentBib}

\end{thebibliography}

\end{document}